\documentclass[10pt]{amsart}

\usepackage{amssymb, dsfont} 
\usepackage{fullpage, graphicx} 
\usepackage{hyperref} 
\usepackage{enumitem} 
\usepackage{caption} 
\usepackage{xcolor} 

\hypersetup{colorlinks=true, urlcolor=purple, citecolor=blue, linkcolor=red, pdfstartview={XYZ null null 0.90}}

\DeclareMathOperator{\Int}{Int} 

\DeclareMathOperator{\trd}{trd} 
\DeclareMathOperator{\nrd}{nrd} 
\DeclareMathOperator{\Emb}{Emb}
\DeclareMathOperator{\discrd}{discrd}
\DeclareMathOperator{\Eich}{Eich}
\DeclareMathOperator{\PB}{PB}

\DeclareMathOperator{\disc}{disc} 
\DeclareMathOperator{\fund}{fund}
\DeclareMathOperator{\Cl}{Cl}
\DeclareMathOperator{\Nm}{Nm}

\DeclareMathOperator{\Mat}{Mat}
\DeclareMathOperator{\SL}{SL}
\DeclareMathOperator{\PSL}{PSL}

\DeclareMathOperator{\Tr}{Tr}
\DeclareMathOperator{\Id}{Id}

\DeclareMathOperator{\sg}{sg}
\DeclareMathOperator{\ord}{ord}
\newcommand{\ZZ}{\ensuremath{\mathbb{Z}}}
\newcommand{\RR}{\ensuremath{\mathbb{R}}}
\newcommand{\QQ}{\ensuremath{\mathbb{Q}}}

\newcommand{\uhp}{\ensuremath{\mathbb{H}}} 
\newcommand{\Guhp}{\ensuremath{\Gamma\backslash\uhp}} 
\newcommand{\Ord}{\ensuremath{\mathrm{O}}}

\newcommand{\sm}[4]{\ensuremath{\left(\begin{smallmatrix} #1 & #2\\#3 & #4\end{smallmatrix}\right)}}
\newcommand{\lm}[4]{\ensuremath{\left(\begin{matrix} #1 & #2\\#3 & #4\end{matrix}\right)}}
\newcommand{\genmtx}{\sm{a}{b}{c}{d}}
\newcommand{\qord}[1]{\ensuremath{\mathcal{O}_{#1}}} 
\newcommand{\qordD}{\qord{D}} 
\newcommand{\nice}{\hyperref[def:nice]{nice}}
\newcommand{\pbad}{\hyperref[def:potentiallybad]{potentially bad}}

\newtheorem{theorem}{Theorem}
\newtheorem{conjecture}[theorem]{Conjecture}
\newtheorem{corollary}[theorem]{Corollary}
\newtheorem{lemma}[theorem]{Lemma}
\newtheorem{proposition}[theorem]{Proposition}

\theoremstyle{definition}

\newtheorem{definition}[theorem]{Definition}
\newtheorem{analogy}[theorem]{Analogy}
\newtheorem{remark}[theorem]{Remark}
\newtheorem{example}[theorem]{Example}

\numberwithin{equation}{section}

\begin{document}

\title[Intersection numbers]{Counting intersection numbers of closed geodesics on Shimura curves}
\author[J. Rickards]{James Rickards}
\address{University of Colorado Boulder, Boulder, Colorado, USA}
\email{james.rickards@colorado.edu}
\urladdr{https://math.colorado.edu/~jari2770/}
\date{\today}
\thanks{This research was supported by an NSERC Vanier Scholarship at McGill University. The author is currently partially supported by NSF-CAREER CNS-1652238 (PI Katherine E. Stange)}
\begin{abstract}
Let $\Gamma\subseteq\text{PSL}(2, \mathbb{R})$ correspond to the group of units of norm $1$ in an Eichler order $\mathrm{O}$ of an indefinite quaternion algebra over $\mathbb{Q}$. Closed geodesics on $\Gamma\backslash\mathbb{H}$ correspond to optimal embeddings of real quadratic orders into $\mathrm{O}$. The weighted intersection numbers of pairs of these closed geodesics conjecturally relates to the work of Darmon-Vonk on a real quadratic analogue to the difference of singular moduli. In this paper, we study the total intersection number over all embeddings of a given pair of discriminants. We precisely describe the arithmetic of each intersection, and produce a formula for the total intersection. This formula is a real quadratic analogue of the work of Gross and Zagier on factorizing the difference of singular moduli. The results are fairly general, allowing for a large class of non-maximal Eichler orders, and non-fundamental/non-coprime discriminants. The paper ends with some explicit examples illustrating the results of the paper.
\end{abstract}
\maketitle

\setcounter{tocdepth}{1}
\tableofcontents

\section{Introduction}
The $\PSL(2, \ZZ)$-invariant $j-$function outputs algebraic numbers given quadratic imaginary inputs. These values generate certain ring class fields, and are known as singular moduli. In the celebrated work of Gross and Zagier in \cite{GZ85}, a formula for the factorization of a difference of singular moduli is given. More concretely, let $D_1, D_2$ be coprime negative fundamental discriminants, and they define the quantity $J(D_1, D_2)^2\in\ZZ$, which is essentially the norm to $\QQ$ of $j(\tau_1)-j(\tau_2)$, for quadratic imaginary $\tau_i$ of discriminant $D_i$. All primes $p$ dividing $J(D_1, D_2)^2$ are shown to satisfy
\[p\mid \dfrac{D_1D_2-x^2}{4}\]
for some integer $x<\sqrt{D_1D_2}$, and the exponents of such primes are computed in terms of Kronecker symbols (for more details, see \cite{GZ85} or Section \ref{sec:connection}). One proof of this result involved counting endomorphisms between elliptic curves, which boiled down to an ``intersection'' computation on definite quaternion algebras. This work was generalized to allow for any distinct negative discriminants $D_1, D_2$ by Lauter and Viray in \cite{LV15}.

On the other hand, a satisfactory analogue of the difference of singular moduli for positive discriminants has remained elusive. A programme begun by Darmon and Vonk in \cite{DV20} is to $p-$adically construct a quantity $J_p(D_1, D_2)$ for positive discriminants $D_1, D_2$, which is conjecturally algebraic and belonging to the compositum of ring class fields associated to $D_1, D_2$. Furthermore, this quantity appears to have a similar factorization to the formula of Gross and Zagier!

The conjectural analogue of $v_q\left(J(D_1, D_2)^2\right)$ is a weighted intersection number of a pair of closed geodesics on a Shimura curve. The aim of this paper is to explore these intersections in as much generality as possible. The main result is Theorem \ref{thm:maincountingthm} (a simplified version is Theorem \ref{thm:countxlinking}), which counts all intersections of geodesics corresponding to a pair of positive discriminants $D_1, D_2$. We work in a fairly general setting, and are also able to count intersections with extra arithmetic information. These formulae are a generalization of the main results in \cite{JR21}, which studied intersection numbers of closed geodesics in $\PSL(2, \ZZ)\backslash\uhp$.

Along the way, we developed algorithms in PARI/GP (\cite{PARI}) to compute (weighted) intersection numbers of closed geodesics (see Section \ref{sec:computational}). These algorithms were crucial in providing evidence towards the connection with the work of Darmon and Vonk, as well as demonstrating and verifying the main counting results of this paper.

\section{Overview of the paper}

This section is dedicated to introducing the setup and notation required to precisely present the main results. A simplified version of the main result is Theorem \ref{thm:countxlinking}, which adds a few conditions to achieve a nicer presentation. A detailed account of the connection to the work of Gross-Zagier and Darmon-Vonk is given in Section \ref{sec:connection}

\subsection{General intersection numbers}

Let $\uhp$ denote the upper half plane and let $\Gamma$ be a discrete subgroup of $\PSL(2,\RR)$, which acts on $\uhp$ via $\genmtx z:=\frac{az+b}{cz+d}$. Geodesics on the orbifold $\Guhp$ are the image of geodesics on $\uhp$, and closed geodesics correspond to elements of $\Gamma$ that are not a non-trivial power of another element of $\Gamma$, and have two distinct real roots. Call such elements \textit{primitive hyperbolic}.

Let $\gamma\in\PSL(2,\RR)$ be a hyperbolic matrix. We label one root to be the first (attracting) root $\gamma_f$, and the other to be the second (repelling) root $\gamma_s$, via the equations
\[\lim_{n\rightarrow\infty}\gamma^n(x):=\gamma_f,\qquad\lim_{n\rightarrow\infty}\gamma^{-n}(x):=\gamma_s,\]
for any $x\in\mathbb{P}^1(\RR)$ that is not a root of $\gamma$. In particular, $\gamma^{-1}$ has the same roots as $\gamma$, but with the first and second roots swapped.

For $z_1,z_2\in\overline{\uhp}:=\uhp\cup\mathbb{P}^1(\RR)$, let $\ell_{z_1,z_2}$ denote the geodesic segment running from $z_1$ to $z_2$, where we do not include the endpoints $z_1,z_2$. Define $\dot{\ell}_{z_1,z_2}:=\ell_{z_1,z_2}\cup\{z_1\}$. If $\gamma\in\Gamma$ is primitive and hyperbolic, define
\[\ell_{\gamma}:=\ell_{\gamma_s,\gamma_f},\]
which is called the \textit{root geodesic} of $\gamma$. For any $z\in\ell_{\gamma}$, the image of $\dot{\ell}_{z,\gamma z}$ in $\Guhp$ is a closed geodesic, denoted by $\tilde{\ell}_{\gamma}$. The image of $\ell_{\gamma}$ in $\Guhp$ runs over $\tilde{\ell}_{\gamma}$ infinitely many times.

If $\gamma_1$ is not conjugate to either $\gamma_2$ or $\gamma_2^{-1}$ in $\Gamma$, then the closed geodesics $\tilde{\ell}_{\gamma_1}$ and $\tilde{\ell}_{\gamma_2}$ intersect in finitely many places. Otherwise, the geodesics completely overlap each other. To get rid of such issues, we refer to \textit{transversal} intersections.

\begin{definition}\label{def:inum}
	Given primitive hyperbolic matrices $\gamma_1,\gamma_2\in\Gamma$, define
	\[\tilde{\ell}_{\gamma_1}\pitchfork\tilde{\ell}_{\gamma_2}\]
	to be the (finite) set of transversal intersections of $\tilde{\ell}_{\gamma_1}$ and $\tilde{\ell}_{\gamma_2}$ in $\Guhp$. Singular points (i.e. having non-trivial stabilizer in $\Gamma$) are counted with multiplicity: fix a local lift of $\tilde{\ell}_{\gamma_1}$, and the multiplicity is the number of local lifts of $\tilde{\ell}_{\gamma_2}$ that intersect transversely with the first lift. If $\tilde{\ell}_{\gamma_1}$ and $\tilde{\ell}_{\gamma_2}$ do not overlap, then this is the size of the stabilizer of the singular point.
	
	Let $f$ be any function defined on transversal intersections. The weighted intersection number of $\gamma_1,\gamma_2$ is defined to be
	\[\Int_{\Gamma}^f(\gamma_1,\gamma_2):=\sum_{z\in\tilde{\ell}_{\gamma_1}\pitchfork\tilde{\ell}_{\gamma_2}}f(z).\]
\end{definition}

In this paper, we consider the \textit{unsigned} intersection number ($f=1$), the \textit{signed} intersection number ($f=$ the sign of intersection), and the $q-$\textit{weighted} intersection number (see below Definition \ref{def:signlevel}).

\subsection{Optimal embeddings}\label{sec:introopt}

Let $B$ be an indefinite quaternion algebra over $\QQ$ of discriminant $\mathfrak{D}$, let $\Ord$ be an Eichler order of level $\mathfrak{M}$, fix an embedding $\iota:B\rightarrow \Mat(2,\RR)$, and let $\qordD$ be the unique quadratic order of discriminant $D$ (lying in $\QQ(\sqrt{D})$). For an integer $r$, define
\[\Ord^r:=\{z\in \Ord:\nrd(z)=r\},\]
the set of elements of reduced norm $r$ in $\Ord$. Note that $\Ord^1$ is a group under multiplication, and let
\[\Gamma=\Gamma_{\Ord}:=\iota(\Ord^1)/\{\pm 1\}\]
be the image of $\Ord^1$ in $\PSL(2, \RR)$, a discrete subgroup.

\begin{definition}
	An \textit{embedding} of $\qordD$ into $\Ord$ is a ring homomorphism $\phi:\qordD\rightarrow\Ord$. Call the embedding \textit{optimal} if it does not extend to an embedding of a larger order into $\Ord$. Note that if $D$ is a fundamental discriminant, then all embeddings of $\qordD$ into $\Ord$ are optimal.
\end{definition}

When $D<0$, call the embedding $\phi$ \textit{positive definite} if $\iota(\phi(\sqrt{D}))_{2,1}>0$ (the lower left entry of the matrix is positive), and \textit{negative definite} otherwise. This notion corresponds to whether the first root of $\iota(\phi(\sqrt{D}))$ (defined similarly to the hyperbolic case) lies in the upper half plane or not. While the individual definitenesses depend on the choice of $\iota$, whether two optimal embeddings of negative discriminants have the same or opposite definiteness is independent of $\iota$.

If $\phi,\phi'$ are optimal embeddings, we define them to be equivalent if there exists $x\in\Ord^1$ such that
\[\phi'=\phi^x:=x\phi x^{-1}.\]
Denote the equivalence class of $\phi$ by $[\phi]$. The notion of equivalence can be extended to pairs of optimal embeddings as follows:
\[(\phi_1,\phi_2)\sim(\phi_1',\phi_2')\text{ if there exists an $x\in\Ord^1$ such that }x\phi_i x^{-1}=\phi_i'\text{ for $i=1,2$.}\]
For a fixed discriminant $D$, define $\Emb(\Ord, D)$ to be the set of equivalence classes of optimal embeddings of $\qordD$ into $\Ord$, which is a finite set (see Proposition \ref{prop:countembD}).

If $D$ is a positive discriminant, let $\epsilon_D>1$ be the fundamental unit with positive norm in $\qordD$. If $\phi$ is an optimal embedding of $\qordD$ into $\Ord$, then $\iota(\phi(\epsilon_D))$ is a primitive hyperbolic element of $\Gamma$ (in fact, all primitive hyperbolic elements of $\Gamma$ arise in this fashion). Define $\ell_{\phi}$ to be $\ell_{\iota(\phi(\epsilon_D))}$.

\begin{definition}
	Let $\phi_1,\phi_2$ be optimal embeddings of positive discriminants $D_1,D_2$, and let $\gamma_i=\iota(\phi_i(\epsilon_{D_i}))$ for $i=1,2$. For any function $f$ defined on transversal intersections, the weighted intersection number of $\phi_1,\phi_2$ is defined to be
	\[\Int_{\Ord}^f(\phi_1,\phi_2):=\Int_{\Gamma_{\Ord}}^f(\gamma_1,\gamma_2).\]
\end{definition}

Note that the intersection number only depends on the equivalence classes of $\phi_1,\phi_2$. In Proposition 1.8 of \cite{JR21}, an alternate interpretation of the intersection number is given. At each transversal intersection, we can lift the point and the local geodesic to the upper half plane. This corresponds to a pair of transversely intersecting root geodesics $\ell_{\sigma_1},\ell_{\sigma_2}$, where $\sigma_i\sim\phi_i$ for $i=1,2$. The obstruction to uniqueness is the choice of lifted intersection point, which is only defined up to $\Gamma$-equivalence. Equivalently, the pair $(\sigma_1,\sigma_2)$ is only defined up to simultaneous equivalence. This is formalized in the following proposition.

\begin{proposition}\label{prop:simulconjinter}
	An alternate interpretation of the intersection number is
	\[\Int_{\Ord}^f(\phi_1,\phi_2)=\sum_{\substack{(\sigma_1,\sigma_2)\in([\phi_1]\times [\phi_2])/\sim \\ \vert\ell_{\sigma_1}\pitchfork\ell_{\sigma_2}\vert=1}}f(\sigma_1,\sigma_2).\]
\end{proposition}

Each intersection point $z$ gives rise to a $\Gamma-$equivalence class of points in $\uhp$, as well as a unique intersection angle, measured from the tangent to $\ell_{\sigma_1}$ at $z$ to the tangent to $\ell_{\sigma_2}$ at $z$.

\subsection{x-linking}
Proposition \ref{prop:simulconjinter} still requires $\iota$ and $\epsilon_D$ to pass to the upper half plane. To make everything contained within the quaternion algebra, we introduce the notion of $x-$linking.

\begin{definition}
	Let $x$ be any integer such that $x^2\neq D_1D_2$. Call the pair $(\phi_1,\phi_2)$ $x-$linked if
	\[x=\dfrac{1}{2}\trd(\phi_1(\sqrt{D_1})\phi_2(\sqrt{D_2})).\]
	In particular, if $(\phi_1,\phi_2)$ is $x-$linked, then every pair in the equivalence class (of simultaneous equivalence) $[(\phi_1,\phi_2)]$ is $x-$linked.
\end{definition}

The case $x^2=D_1D_2$ is a degenerate case, and will not be relevant here. If $(\phi_1,\phi_2)$ are $x-$linked, then $x\equiv D_1D_2\pmod{2}$. Whether two optimal embeddings intersect is completely determined by their $x-$linking, as demonstrated in the following proposition (proven in Section \ref{sec:basicresults}).

\begin{proposition}\label{prop:optembint}
	Assume that $(\phi_1,\phi_2)$ are $x-$linked optimal embeddings of positive discriminants $D_1,D_2$. Then the root geodesics $\ell_{\phi_1},\ell_{\phi_2}$ intersect transversely if and only if 
	\[x^2<D_1D_2.\]
	In this case,
	\begin{enumerate}[label=(\roman*)]
		\item The intersection point is the upper half plane root of $\iota(\phi_1(\sqrt{D_1})\phi_2(\sqrt{D_2}))$, and so it corresponds to an (not necessarily optimal) embedding of the negative quadratic order $\qord{x^2-D_1D_2}$.
		\item The angle of intersection $\theta$ satisfies
		\[\cos(\theta)=\dfrac{x}{\sqrt{D_1D_2}}.\]
	\end{enumerate}
\end{proposition}

Define
\[\Emb(\Ord,\phi_1,\phi_2,x):=\{(\sigma_1,\sigma_2): \sigma_1\sim\phi_1, \sigma_2\sim\phi_2, (\sigma_1,\sigma_2)\text{ are $x-$linked}\}/\sim,\]
the equivalence classes of $x-$linked pairs of embeddings similar to $\phi_1,\phi_2$. Going further, write
\begin{align*}
	\Emb(\Ord,D_1,D_2,x):= & \{(\sigma_1,\sigma_2): [\sigma_i]\in\Emb(\Ord,D_i), (\sigma_1,\sigma_2)\text{ are $x-$linked}\}/\sim\\
	= & \bigcup_{[\phi_i]\in\Emb(\Ord,D_i)}\Emb(\Ord,\phi_1,\phi_2,x).
\end{align*}

In particular, the intersection number can be rephrased without reference to $\iota$ or the fundamental units as follows:
\[\Int_{\Ord}^f(\phi_1,\phi_2)=\sum_{\substack{x^2<D_1D_2 \\ x\equiv D_1D_2\pmod{2}}}\sum_{[(\sigma_1,\sigma_2)]\in\Emb(\Ord,\phi_1,\phi_2,x)}f(\sigma_1,\sigma_2).\]

Thus, an intersection of $\phi_1$ with $\phi_2$ can be thought of as an $x-$linked pair $(\sigma_1,\sigma_2)$, with $\vert x\vert<\sqrt{D_1D_2}$ and $\sigma_i\sim\phi_i$.

\begin{definition}\label{def:signlevel}
	Let $\sigma_1\times\sigma_2$ denote the unique optimal embedding which satisfies
	\[\sigma_1\times\sigma_2(x+\sqrt{x^2-D_1D_2})=\sigma_1(\sqrt{D_1})\sigma_2(\sqrt{D_2}).\]
	The \textit{sign} of the intersection $(\sigma_1,\sigma_2)$, denoted $\sg(\sigma_1,\sigma_2)$, is $1$ if $\sigma_1\times\sigma_2$ is positive definite, and $-1$ otherwise (it is left undefined if $x^2>D_1D_2$). The \text{level} of the intersection, denoted $\ell(\sigma_1,\sigma_2)$, is $\ell>0$ where $\sigma_1\times\sigma_2$ is an optimal embedding of discriminant $\frac{x^2-D_1D_2}{\ell^2}$.
	
	Define $\Emb(\Ord,D_1,D_2,x,\ell)$ to be the set of pairs of intersections in $\Emb(\Ord,D_1,D_2,x)$ that have level $\ell$, and if $x^2<D_1D_2$, define $\Emb^+(\Ord,D_1,D_2,x,\ell)$ to be the subset of pairs that also have positive sign.
\end{definition}

Using the notion of sign and level, we can describe three different intersection functions $f$
\begin{enumerate}
	\item When $f(\sigma_1,\sigma_2)=1$, $\Int_{\Ord}^f$ is called the unsigned intersection number, and is denoted $\Int_{\Ord}$.
	\item When $f(\sigma_1,\sigma_2)=\sg(\sigma_1,\sigma_2)$, $\Int_{\Ord}^f$ is called the signed intersection number, and is denoted $\Int_{\Ord}^{\pm}$.
	\item When $q$ is a prime and $f(\sigma_1,\sigma_2)=\sg(\sigma_1,\sigma_2)(1+v_q(\ell(\sigma_1,\sigma_2)))$, $\Int_{\Ord}^f$ is called the $q-$weighted intersection number, and is denoted $\Int_{\Ord}^{q}$.
\end{enumerate}

\begin{remark}
	The intersection sign can equivalently be defined as the topological intersection sign of the corresponding root geodesics.
\end{remark}

\subsection{Main result}
The $\epsilon$ function defined in Gross and Zagier (\cite{GZ85}) is very important for $x-$linking. We recall its definition here (see Definition \ref{def:epsilon} for a slight generalization).

\begin{definition}\label{def:epsilonoriginal}
	Let $D_1,D_2$ be coprime fundamental discriminants, and let $p$ be a prime for which $\left(\frac{D_1D_2}{p}\right)\neq -1$. Define
	\[
	\epsilon(p):=\begin{cases} \left(\dfrac{D_1}{p}\right) & \text{ if $p$ and $D_1$ are coprime;}\\\\
		\left(\dfrac{D_2}{p}\right) & \text{ if $p$ and $D_2$ are coprime.}
	\end{cases}
	\]
\end{definition}

Note that $\epsilon$ is well defined if $p\nmid D_1D_2$ (as $1=\left(\frac{D_1D_2}{p}\right)$), and it is defined on all prime factors of $\frac{D_1D_2-x^2}{4}$. 

\begin{sloppypar} The ``holy grail'' of counting intersection numbers would be to identify the constituent terms in $\Emb^+(\Ord, \phi_1, \phi_2, x, \ell)$, though this seems nonviable (at least with the current approach). Thus, we settle for the more general term, where we replace the embedding $\phi_i$ with its discriminant $D_i$. A simplified version of our main result is the following theorem.
\end{sloppypar}

\begin{theorem}\label{thm:countxlinking}
	Let $D_1,D_2$ be positive coprime fundamental discriminants, and let $x$ be any integer such that $x\equiv D_1D_2\pmod{2}$. Then there is precisely one quaternion algebra $B$ over $\QQ$ which contains a maximal order $\Ord$ such that there exists $x-$linked optimal embeddings from $\qord{D_i}$ into $O $. Furthermore, factorize
	\[\dfrac{D_1D_2-x^2}{4}=\pm\prod_{i=1}^{r}p_i^{2e_i+1}\prod_{i=1}^s q_i^{2f_i}\prod_{i=1}^t w_i^{g_i},\]
	where the $p_i$ are the primes for which $\epsilon(p_i)=-1$ that appear to an odd power, $q_i$ are the primes for which $\epsilon(q_i)=-1$ that appear to an even power, and $w_i$ are the primes for which $\epsilon(w_i)=1$. If $B$ has discriminant $\mathfrak{D}$, then:
	\begin{enumerate}[label=(\roman*)]
		\item $r$ is even and $\mathfrak{D}=\prod_{i=1}^r p_i$.
		\item The size of $\Emb(\Ord,D_1,D_2,x)$ is $2^{r+1}\prod_{i=1}^t(g_i+1)$.
		\item The set $\Emb(\Ord,D_1,D_2,x,\ell)$ is non-empty if and only if 
		\[\ell=\prod_{i=1}^{r}p_i^{e_i}\prod_{i=1}^s q_i^{f_i}\prod_{i=1}^t w_i^{g_i'},\]
		where $2g_i'\leq g_i$.
		\item Assume the above holds, and let $n$ be the number of indices for which $2g_i'<g_i$. Then 
		\[\vert\Emb(\Ord,D_1,D_2,x,\ell)\vert=2^{r+n+1}.\]
		If $x^2<D_1D_2$, then exactly half of these embeddings have positive sign.
	\end{enumerate}
\end{theorem}

Theorem \ref{thm:maincountingthm} is a generalization of this result, where we allow for Eichler orders, drop the requirement of fundamentalness, and weaken the coprimality condition. By adding an additional assumption, in Corollary \ref{cor:countembwitheverything} we also consider orientations (see Section \ref{sec:optemb}) of the optimal embedding pairs in $\Emb^+(\Ord,D_1,D_2,x,\ell)$.

\subsection{Connection to other work}\label{sec:connection}

In \cite{GZ85}, Gross and Zagier take $D_1, D_2$ to be negative fundamental coprime discriminants, and define the integral quantity $J(D_1, D_2)^2$, which is essentially the norm to $\QQ$ of $j(\tau_1)-j(\tau_2)$. Their Theorem 1.3 says
\[J(D_1, D_2)^2=\pm\prod_{\substack{x^2<D_1D_2 \\ x\equiv D_1D_2\pmod{2}}} F_{\text{GZ}}\left(\dfrac{D_1D_2-x^2}{4}\right),\]
for the function
\[F_{\text{GZ}}(m)=\prod_{nn'=m, n>0}n^{\epsilon(n')}.\]

Let $D_1, D_2$ be positive coprime fundamental discriminants, and let $\Ord$ be a maximal order of an indefinite quaternion algebra $B$ of discriminant $\mathfrak{D}$ over $\QQ$. A consequence of Proposition \ref{prop:optembint} and Theorem \ref{thm:countxlinking} is that the total unsigned intersection of discriminants $D_1, D_2$ into $\Ord$ is
\begin{align*}
	\Int_{\Ord}(D_1,D_2):= & \sum_{[\phi_1]\in\Emb(D_1, \Ord)}\sum_{[\phi_2]\in\Emb(D_2, \Ord)}\Int_{\Ord}(\phi_1,\phi_2)\\
	= & \sum_{\substack{x^2<D_1D_2 \\ x\equiv D_1D_2\pmod{2}}} F\left(\dfrac{D_1D_2-x^2}{4}\right).
\end{align*}
In particular, this naturally takes the exact same form (albeit with product replaced by sum). Taking the factorization as in Theorem \ref{thm:countxlinking}, we have
\begin{itemize}
	\item $F\left(\frac{D_1D_2-x^2}{4}\right)\neq 0$ if and only if $\mathfrak{D}=\prod_{i=1}^{r}p_i$;
	\item If this holds, then
	\[F\left(\dfrac{D_1D_2-x^2}{4}\right)=2^{r+1}\prod_{i=1}^t(g_i+1)=2^{r+1}\sum_{d\mid\frac{D_1D_2-x^2}{4\mathfrak{D}}} \epsilon(d).\]
\end{itemize}
On the Gross-Zagier side, take the same factorization as above, and take $\ell$ to be a prime. Then
\begin{itemize}
	\item $v_{\ell}\left(F_{\text{GZ}}\left(\frac{D_1D_2-x^2}{4}\right)\right)\neq 0$ if and only if $\ell=\prod_{i=1}^{r}p_i$ (i.e. $r=1$ and $p_1=\ell$);
	\item If this holds, then
	\[v_{\ell}\left(F_{\text{GZ}}\left(\frac{D_1D_2-x^2}{4}\right)\right)=(e_1+1)\prod_{i=1}^t(g_i+1)=(e_1+1)\sum_{d\mid\frac{D_1D_2-x^2}{4\ell}} \epsilon(d).\]
\end{itemize}

In particular, this cements the analogy between the two situations.
\begin{analogy}
	The total intersection number of positive discriminants, $\Int_{\Ord}(D_1,D_2)$, behaves like the exponents of primes in the factorization of $J(D_1, D_2)^2$ for negative discriminants.
\end{analogy}

The individual intersection numbers $\Int_{\Ord}(\phi_1,\phi_2)$ should then have an analogy involving the exponents of primes above $\ell$ in the factorization of $j(\tau_1)-j(\tau_2)$. To make such a connection concrete, we require a real quadratic analogue of $j(\tau_1)-j(\tau_2)$, and not just the exponents. This connection is the goal of Darmon-Vonk in \cite{DV20}.

In this work, given $\tau_1,\tau_2$ real quadratic points corresponding to coprime fundamental discriminants $D_1, D_2$ and a prime $p\leq 13$, they construct a $p$-adic quantity $J_p(D_1, D_2)$, which is conjecturally algebraic and belonging to the compositum of ring class fields associated to $D_1, D_2$.

\begin{conjecture}[Conjecture 4.26 of \cite{DV20}]\label{conj:DV}
	Let $\mathfrak{q}$ lie above the integer prime $q\neq p$. If $q$ is split in $\QQ(\sqrt{D_1})$ or $\QQ(\sqrt{D_2})$, then $\ord_{\mathfrak{q}}(J_p(\tau_1,\tau_2))=0$. Otherwise, let $\Ord$ be a maximal order in the quaternion algebra ramified at $p,q$. Then there exist optimal embeddings $\phi_1,\phi_2$ of discriminants $D_1,D_2$ into $\Ord$ for which
	\[\ord_{\mathfrak{q}}(J_p(\tau_1,\tau_2))=\Int_{\Ord}^q(\phi_1,\phi_2).\]
	In other words, the exponents of primes above $q$ in the factorizations of $J_p(\tau_1,\tau_2)$ are given by $q-$weighted intersection numbers associated to optimal embeddings of $D_1,D_2$ into a maximal order in the indefinite quaternion algebra ramified at $p,q$.
\end{conjecture}

Besides the compelling analogy between Gross-Zagier, Darmon-Vonk, and this work, we have extensive computational evidence. I computed the intersection numbers $\Int_{\Ord}^q(\phi_1,\phi_2)$ for all pairs with $D_1=5,13$ and $D_2\leq 1000$, and compiled it into a 600 page document. On the other side, Jan Vonk computed the $q-$adic valuations of $J_p(\tau_1,\tau_2)$ for many of these examples, and the data matched perfectly.

\subsection{Computational aspects}\label{sec:computational}

Everything described in this paper has been implemented by the author in PARI/GP (\cite{PARI}), and the corresponding package can be found on GitHub at \cite{Qquad}. In particular, this includes algorithms to:
\begin{itemize}
	\item Initialize a quaternion algebra $B$ over $\QQ$ of a specified ramification, as well as an Eichler order $\Ord$ of a given level;
	\item Compute representatives of the equivalence classes in $\Emb(\Ord, D)$, divide them into classes by their orientation, and sort these classes by the action of $\Cl^+(D)$;
	\item Compute the sets $\Emb(\Ord, \phi_1, \phi_2, x)$, and the corresponding signs and levels.
	\item Compute all non-trivial unsigned, signed, and $q-$weighted intersection numbers of a given pair of discriminants $D_1, D_2$.
\end{itemize}

As mentioned in the last section, these computations were essential to establishing the connection to the work of Darmon and Vonk.

\subsection{Plan of attack}
Section \ref{sec:qback} recalls and proves some basic results on quaternion algebras that will be useful later. Section \ref{sec:basicresults} covers some basic results on intersection numbers. Section \ref{sec:xlink} studies the conditions on which there exist $x-$linked optimal embeddings of a given pair of discriminants. In Section \ref{sec:countEO}, we count the Eichler orders containing a given pair of $x-$linked embeddings. Section \ref{sec:mainproof} assembles all of the ingredients to prove the generalization of Theorem \ref{thm:countxlinking}. The paper ends by providing some explicit examples demonstrating the main results.

\section{Quaternionic background}\label{sec:qback}
In this section we recall properties of quaternion algebras and Eichler orders that are required in Sections \ref{sec:basicresults} and beyond. The main focus will be on optimal embeddings. For a full exposition on quaternion algebras, see \cite{JV21}.

\subsection{Local and global quaternion algebras}
Let $F$ be a field of characteristic $0$, and $a,b\in F^{\times}$. Take $B=\left(\frac{a,b}{F}\right)$ to be the quaternion algebra associated to $a,b,F$. As an additive vector space, this is of dimension $4$ over $F$, with basis $1,i,j,k$, and general element of the form
\[x=e+fi+gj+hk,\text{ where }e,f,g,h\in F.\]
The multiplicative structure is determined by the standard equations
\[i^2=a,\qquad j^2=b,\qquad k=ij=-ji.\]
The standard involution on $B$ is denoted by an overline, and explicitly defined by
\[\overline{x}:=e-fi-gj-hk.\]
The quaternion algebra also comes equipped with the reduced trace $\trd:B\rightarrow F$ and the reduced norm $\nrd:B\rightarrow F$, defined by
\begin{align*}
	\nrd(x):= & x\overline{x}=e^2-af^2-bg^2+abh^2;\\
	\trd(x):= & x+\overline{x}=2e.
\end{align*}

When $F=\RR$, there are exactly two quaternion algebras up to isomorphism: $\Mat(2,\RR)$, and the Hamilton quaternions $\left(\frac{-1,-1}{\RR}\right)$ (which is a division algebra). Similarly, over $\QQ_p$, there are two quaternion algebras up to isomorphism: $\Mat(2,\QQ_p)$, and a division algebra. The division algebra can be written as $\left(\frac{p,e}{\QQ_p}\right)$, where $e$ is any integer such that $\left(\frac{e}{p}\right)=-1$, and $\left(\frac{\cdot}{p}\right)$ is the Kronecker symbol.

Let $B=\left(\frac{a,b}{\QQ}\right)$ be a quaternion algebra over $\QQ$. Much of the structure of $B$ is determined by its local behaviour, i.e. the local quaternion algebras $B_v=B\otimes\QQ_v=\left(\frac{a,b}{\QQ_v}\right)$, where $v$ is a place of $\QQ$ and $\QQ_{\infty}=\RR$. Call $v$ \textit{ramified} in $B$ if $B_v$ is division, and call $v$ \textit{split} otherwise. Define the Hilbert symbol $(a,b)_v$ to be $1$ if $p$ is split in $B$, and $-1$ if $B$ is ramified. The set of ramified places is both finite and of even size, and we say that $B$ has discriminant $\mathfrak{D}$, where $\mathfrak{D}$ is the product of all ramifying places.

The quaternion algebra $B$ over $\QQ$ is uniquely determined (up to isomorphism) by the set of ramifying places, and furthermore, any finite even sized set of places corresponds to a quaternion algebra over $\QQ$. We call $B$ indefinite if $\infty$ is split, hence $B$ is ramified at an even number of finite primes. We will generally be working with indefinite quaternion algebras over $\QQ$, although some results work in more generality.

An order $\Ord$ of $B$ is a lattice that is also a subring. A maximal order is an order which is not properly contained within another order. All maximal orders of $\Mat(2,\QQ_p)$ are conjugate, whereas the division quaternion algebra over $\QQ_p$ has a unique maximal order, consisting of all integral elements. Globally, all maximal orders in an indefinite quaternion algebra over $\QQ$ are conjugate.

If $F=\QQ,\QQ_p$, then an order $\Ord$ is always a dimension four $\mathcal{O}_F=\ZZ,\ZZ_p-$module (respectively). Let $\alpha_1,\alpha_2,\alpha_3,\alpha_4$ be a basis of $\Ord$, and define the discriminant of $\Ord$ to be
\[\disc(\Ord)=-d(\alpha_1,\alpha_2,\alpha_3,\alpha_4):=-\det(\trd(\alpha_i\alpha_j)_{i,j}).\]
This is always a square, and the reduced discriminant of $\Ord$ is defined by
\[\discrd(\Ord)^2=\disc(\Ord).\]
The reduced discriminant is only defined up to $\mathcal{O}_F^{\times}$, so over $\QQ$ we take the convention that it is positive, and over $\QQ_p$ we take it to be of the form $p^e$ with $e\geq 0$. It follows that if $F=\QQ$, then
\[\discrd(\Ord)=\prod_p\discrd(\Ord_p),\]
where the product is taken over all primes $p$ and $\Ord_p=\Ord\otimes\ZZ_p$ is the corresponding local order in $B_p$.

Over $F=\QQ$, an order is maximal if and only if its reduced discriminant is equal to the finite part of $\mathfrak{D}$, the discriminant of the quaternion algebra. A general order $\Ord$ will have $\discrd(\Ord)=\mathfrak{D}\mathfrak{M}$, where $\mathfrak{M}$ is called the level of the order.

Working locally will be essential, so we will state the local-global correspondence for lattices (which also holds for orders in a quaternion algebra).

\begin{theorem}[Variant of Theorem 9.1.1 of \cite{JV21}]\label{thm:JVlocalglobal}
	Let $V$ be a finite dimensional $\QQ-$vector space, and let $M\subset V$ be a $\ZZ$-lattice. Then the map $N\rightarrow (N_p)_{p}$ gives a bijection between $\ZZ-$lattices $N\subset V$ and collections of $\ZZ_p$-lattices $(N_p)_{p}$ indexed by the primes which satisfy $M_p=N_p$ for all but finitely many primes $p$.
\end{theorem}

An Eichler order $\Ord$ of $B$ is an order that is the intersection of two (uniquely determined) maximal orders. Over $F=\QQ_p$, if $B$ is division there is exactly one maximal order, hence this is the only Eichler order. Otherwise, $B=\Mat(2,\QQ_p)$, and there exist Eichler orders of levels $p^e$ for all $e\geq 0$. They are all conjugate, and we define the standard Eichler order of level $p^e$ to be
\[\lm{\ZZ_p}{\ZZ_p}{p^e\ZZ_p}{\ZZ_p}.\]
Following the local-global principle of orders, when $F=\QQ$, an order $\Ord$ is Eichler if and only if $\Ord_p$ is Eichler for all primes $p$.

Furthermore, if $B$ is indefinite, a consequence of strong approximation is that all Eichler orders of the same level are conjugate over $B^{\times}$.

\subsection{Normalizer of an Eichler order}
Take $B$ to be a quaternion algebra over $F=\QQ$ or $F=\QQ_p$.

\begin{definition}
	Let $\Ord$ be an order in $B$, and define the subgroup of $x\in B^{\times}$ for which $x\Ord x^{-1}=\Ord$ to be $N_{B^{\times}}(\Ord)$, the normalizer group of the order $\Ord$.
\end{definition}

Clearly $F^{\times}\Ord^{\times}\subseteq N_{B^{\times}}(\Ord)$. As we will see in Proposition \ref{prop:localstab}, this is a finite index subgroup.

\begin{lemma}\label{lem:eltstab}
	Let $x\in B-F$. Then the set $C_{B}(x):=\{v\in B: vx=xv\}$ is an $F-$algebra and a two dimensional $F-$vector space spanned by $1,x$. We call it the centralizing algebra of $x$.
\end{lemma}
\begin{proof}
	This follows immediately from Proposition 7.7.8 of \cite{JV21}.
\end{proof}

\begin{corollary}\label{cor:eltconj}
\begin{sloppypar}	Let $x_1,x_2\in B^{\times}-F^{\times}$ have the same separable minimal polynomial. Then the set $C_{B}(x_1,x_2):=\{v\in B: vx_1=x_2v\}$ is a two dimensional $F-$vector space.
\end{sloppypar}
\end{corollary}
\begin{proof}
	By Corollary 7.7.3 of \cite{JV21}, the equality of the minimal polynomials of $x_1,x_2$ implies that there exists a $w\in B^{\times}$ with $wx_1w^{-1}=x_2$. Thus
	\[v\in C_B(x_1,x_2)\Leftrightarrow vw^{-1} x_2= x_2 vw^{-1},\]
	so the corollary follows from Lemma \ref{lem:eltstab}.
\end{proof}

We now describe the normalizer groups of Eichler orders over $\QQ_p$.

\begin{proposition}\label{prop:localstab}
	Let $B$ be a quaternion algebra over $\QQ_p$ with Eichler order $\Ord$. If $B$ is division, we have
	\[N_{B^{\times}}(\Ord)=B^{\times}.\]
	Otherwise, write $B=\Mat(2,\QQ_p)$ and take $\Ord=\sm{\ZZ_p}{\ZZ_p}{p^e\ZZ_p}{\ZZ_p}$. Let $\omega:=\sm{0}{1}{-p^e}{0}$, and then
	\[N_{B^{\times}}(\Ord)=\QQ_p^{\times}\Ord^{\times}\langle\omega\rangle.\]
\end{proposition}
\begin{proof}
	If $B$ is division, then $\Ord$ is the unique maximal order. Since conjugates of $\Ord$ are also maximal orders, it is stabilized under conjugation by all of $B^{\times}$. When $B$ is not division, this is Proposition 23.4.14 of \cite{JV21} (the definition of $\omega$ has been adjusted so that it has positive norm). 
\end{proof}

Note that if $B$ is not division and $\Ord$ is maximal, then $N_{B^{\times}}(\Ord)=\QQ_p^{\times}\Ord^{\times}$. Translating the above proposition into the global case yields the following proposition.

\begin{proposition}\label{prop:stab}
	Let $B$ be an indefinite quaternion algebra over $\QQ$ with discriminant $\mathfrak{D}$, and let $\Ord$ be an Eichler order of $B$ of level $\mathfrak{M}$. Then there exists a collection of elements $\{\omega_p: p\mid\mathfrak{D}\mathfrak{M}\infty\}$ with $\nrd(\omega_p)=p^{v_p(\mathfrak{D}\mathfrak{M})}$ for $p<\infty$ and $\nrd(\omega_{\infty})=-1$ for which
	\[\dfrac{N_{B^{\times}}(\Ord)}{\QQ^{\times}\Ord^1}=\langle\omega_p\rangle_{p\mid\mathfrak{D}\mathfrak{M}\infty}\simeq\prod_{p\mid\mathfrak{D}\mathfrak{M}\infty}\dfrac{\ZZ }{2\ZZ }.\]
\end{proposition}
\begin{proof}
	By combining Proposition 18.5.3 and Equation 23.4.20 of \cite{JV21} with the fact that $\Ord$ has class number one, we get the isomorphism
	\[\dfrac{N_{B^{\times}}(\Ord)}{\QQ ^{\times}\Ord^{\times}}\simeq\prod_{p\mid\mathfrak{D}\mathfrak{M}}\dfrac{\ZZ}{2\ZZ}.\]
	By taking a set of generators and looking locally, we can use Proposition \ref{prop:localstab} to show that we can find an equivalent set of generators $\{\omega_p\}_{p\mid\mathfrak{D}\mathfrak{M}}$ which satisfy $\nrd(\omega_p)=p^{v_p(\mathfrak{D}\mathfrak{M})}$ for $p<\infty$. Finally, we can pull out the $\infty$ by using $\Ord^{\times}=\Ord^1\cup\Ord^{-1}$, and $\Ord^{-1}=\omega_\infty\Ord^1$ for any $\omega_{\infty}\in\Ord^{-1}$.
\end{proof}

\subsection{Towers of Eichler orders}\label{sec:towersEichler}
The Bruhat-Tits tree provides a combinatorial aspect to the theory of maximal/Eichler orders of $B=\Mat(2,\QQ_p)$. Vertices of the graph are maximal orders in $B$, and there exists an edge between $\Ord$ and $\Ord'$ if and only if $\Ord\cap\Ord'$ is an Eichler order of level $p$. A summary of the main facts of the graph (see Section 23.5 of \cite{JV21}) are:
\begin{itemize}
	\item The graph is connected and has no cycles, hence it is a tree (as the name implies);
	\item Every vertex has degree $p+1$;
	\item Let $\Ord_1,\Ord_2$ be maximal orders, and let $\Ord=\Ord_1\cap\Ord_2$ be the corresponding Eichler order of level $p^e$. Then $\Ord$ corresponds to the unique path between $\Ord_1$ and $\Ord_2$. This path has length $e$, and the vertices on the path are precisely the $e+1$ maximal orders which contain $\Ord$.
\end{itemize}

Focusing on one Eichler order $\Ord$ of level $p^e$, we define the ``inverted triangle'' of superorders of $\Ord$ as follows:
\begin{itemize}
	\item It is a graph consisting of all (necessarily Eichler) superorders $\Ord'\supseteq\Ord$ as vertices;
	\item The vertices are arranged into $e+1$ rows, where the $i^{\text{th}}$ row from the top (starting with row $0$) consists of the Eichler orders of level $p^i$ containing $\Ord$.
	\item There is an edge between orders $\Ord_1,\Ord_2$ if and only if one order contains the other and they are in adjacent rows.
\end{itemize}
It follows directly from the Bruhat-Tits tree that there are $e+1-i$ vertices in the $i^{\text{th}}$ row, and the graph can be drawn in the plane so that each vertex (besides those in row $0$) is connected to the two closest vertices in the row above it. An Eichler order is the intersection of the two orders it is connected to in the above row. As an example, the inverted triangle for an Eichler order of level $p^5$ is displayed in Figure \ref{fig:invertedtriangle}.

\begin{figure}[htbp]
	\centering
	\includegraphics{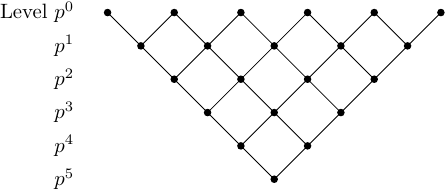}
	\caption{Inverted triangle of level $p^5$.}\label{fig:invertedtriangle}
\end{figure}

The inverted triangle of $\Ord$ allows one to count superorders of $\Ord$ of a specified level which do not contain certain given superorders (which is required in Section \ref{sec:localxwithlevel}). 

\begin{remark}
	The inverted triangle of $\Ord$ is essentially the same concept as branches of orders, as found in \cite{AC13} and \cite{AAC18}.
\end{remark}

\subsection{Optimal embeddings}\label{sec:optemb}
Let $B$ be a quaternion algebra over $F=\QQ$ or $F=\QQ_p$, and let $\Ord$ be an order in $B$. If $F=\QQ_p$, we call an embedding of $\qordD$ into $\Ord$ optimal if it does not extend to an embedding of $\qord{D/p^2}$ (which is automatic if $D/p^2$ is not a discriminant). In particular, if $F=\QQ$, an embedding $\phi$ into $\Ord$ is optimal if and only if the corresponding embeddings $\phi_p$ into $\Ord_p$ are optimal for all primes $p$. 

\begin{definition}
	For a discriminant $D$, define $p_D\in\{0,1\}$ to be the parity of $D$. Let the field discriminant of $\QQ(\sqrt{D})$ be $D^{\fund}$.
\end{definition}

Since $\qordD=\ZZ \left[\frac{p_D+\sqrt{D}}{2}\right]$, an embedding of $\qordD$ into $\Ord$ is equivalent to picking an element $x=\phi\left(\frac{p_D+\sqrt{D}}{2}\right)\in\Ord$ which has the same characteristic polynomial as $\frac{p_D+\sqrt{D}}{2}$, i.e. an element $x$ satisfying $x^2-p_Dx+\frac{p_D-D}{4}=0$.

In certain proofs, it will be useful to assume that an optimal embedding takes a certain form. Corollary \ref{cor:optembgoestoi} allows us to do this.

\begin{lemma}[Exercise 2.5 of \cite{JV21}]\label{lem:wlogxisi}
	Let $B$ be a quaternion algebra over a field $F$ of characteristic not equal to $2$, and assume $x\in B\backslash F$ satisfies $x^2=n\in F^{\times}$. Then there exists an $m\in F^{\times}$ and an isomorphism $\theta:B\rightarrow\left(\frac{n,m}{F}\right)$ satisfying $\theta(x)=i$.
\end{lemma}
\begin{proof}
	Consider the inner product defined as $\langle u, v\rangle=\frac{1}{2}\trd(u\overline{v})$. Pick any $y$ such that $B$ is generated as an $F-$algebra by $x,y$, and by applying the Gram-Schmidt orthogonalization process, we can assume that $0=\langle 1, y\rangle=\langle x, y\rangle$. This implies that $y^2=m\in F^{\times}$ and $xy=-yx$, whence we have the result.
\end{proof}

\begin{corollary}\label{cor:optembgoestoi}
	Let $\phi:\qordD\rightarrow\Ord$ be an (optimal) embedding into an order of the quaternion algebra $B$. Then there exists a quaternion algebra $B'$ with order $\Ord'$ and an isomorphism $\theta:B\rightarrow B'$ taking $\Ord$ to $\Ord'$ such that $\theta\circ\phi:\qordD\rightarrow\Ord'$ is an (optimal) embedding with $\theta\circ\phi(\sqrt{D})=i_{B'}$. In particular, given an (optimal) embedding, we can choose coordinates so that the image of $\sqrt{D}$ is $i$.
\end{corollary}
\begin{proof}
	Take $x=\phi(\sqrt{D})$ in Lemma \ref{lem:wlogxisi}, and consider the corresponding map $\theta$. Let $\Ord'=\theta(\Ord)$, and then $\Ord'$ is an isomorphic order for which $\theta\circ\phi$ is an (optimal) embedding into.
\end{proof}

We would like to count the set $\Emb(\Ord, D)$, and Chapter 30 of \cite{JV21} provides an excellent exposition of this in a more general context. We now restate the relevant results in our setting, and expand upon the notion of equivalence classes of the localized embeddings (which we refer to as orientation). If $\Ord$ is an Eichler order in a quaternion algebra $B$ over $\QQ_p$ or $\RR$ ($\Ord=B$ if over $\RR$), define $\Emb(\Ord, D)$ analogously to over $\QQ$ (Section \ref{sec:introopt}).

\begin{proposition}\label{prop:countembD}
	Let $D$ be a discriminant, let $B$ an indefinite quaternion algebra over $\QQ$, and let $\Ord$ an Eichler order. Let $h^+(D)$ denote the narrow class number of discriminant $D$. Then,
	\[\vert\Emb(\Ord, D)\vert=h^+(D)\prod_{v}\vert\Emb(\Ord_v, D)\vert,\]
	where the product is over all places of $\QQ$.
\end{proposition}
\begin{proof}
	The class number of any Eichler order over $\QQ$ is one, and the result then follows from Theorem 30.7.3 of \cite{JV21} and the surrounding results. See also Sections 4.4 and 4.5 of \cite{JRthe} for an alternate presentation.
\end{proof}

In particular, $\vert\Emb(\Ord, D)\vert$ is $h^+(D)$ up to local factors. The local factors are as follows.

\begin{proposition}\label{prop:countlocalfactors}
	Let $D$ be a discriminant, and let $B$ a quaternion algebra over $\RR$ or $\QQ_p$.
	\begin{enumerate}[label=(\roman*)]
		\item If $B=\Ord=\Mat(2,\RR)$, then 
		\[\vert\Emb(\Ord,D)\vert=1+\mathds{1}_{D<0}.\]
		\item If $B$ is division over $\QQ_p$ with maximal order $\Ord$, then
		\[\vert\Emb(\Ord,D)\vert=\begin{cases}
			0 & \text{if }p^2\mid\frac{D}{D^{\fund}};\\
			1-\left(\dfrac{D}{p}\right) & \text{else}.
		\end{cases}\]
		\item If $B=\Mat(2,\QQ_p)$, $\Ord$ is an Eichler order of level $p^e$, and $\gcd(p^e, D)=1$, then
		\[\vert\Emb(\Ord,D)\vert=\begin{cases}
			1 & \text{if }e=0;\\
			1+\left(\dfrac{D}{p}\right) & \text{if }e>0.
		\end{cases}\]
	\end{enumerate}
\end{proposition}
\begin{proof}
	The first part follows easily from the Skolem-Noether theorem. The second part follows from Proposition 30.5.3 of \cite{JV21} in the case of $p^2\nmid\frac{D}{D^{\fund}}$. Otherwise, since $\Ord$ is the set of all integral elements in $B$, any embedding of $\qordD$ extends to an embedding of $\qord{D^{\fund}}$. The third part follows from Propositions 30.5.3 and 30.6.12 of \cite{JV21}.
\end{proof}

The above proposition omits the case of $B=\Mat(2,\QQ_p)$, $\Ord$ is an Eichler order of level $p^e$ with $e>0$ and $p\mid D$. This case is much more complicated, and its description will not be of use to us. If desired, see Lemma 30.6.17 of \cite{JV21} for the details.

\begin{definition}
	Assume $B$ is indefinite over $\QQ$, and let $\phi$ be an optimal embedding into an Eichler order $\Ord$. For all places $v$, let $o_v(\phi)$ denote the local equivalence class of $\phi_v$. The orientation of $\phi$ is defined to be
	\[o(\phi):=(o_v(\phi))_v:\text{$v$ is a place},\]
	the set of equivalence classes of the corresponding local embeddings.
	
	If $\gcd(D,\mathfrak{M})=1$, then all local embedding equivalence classes have size either $1$ or $2$. In particular, write $o_v(\phi)=0$ or $o_v(\phi)=\pm 1$ for the one or two local equivalence classes (this is non-canonical and depends on an initial choice when there are two local classes). 
\end{definition}

\begin{definition}
	For each orientation $o$ of an optimal embedding of $\qordD$ into $\Ord$, we denote by $\Emb_o(\Ord,D)$ the equivalence classes of optimal embeddings with orientation $o$.
\end{definition}

Note that we can restrict the orientation to places $p\mid\mathfrak{D}\mathfrak{M}\infty$, since Proposition \ref{prop:countlocalfactors} implies that there is one local orientation at all other places. At those places, it will be useful to have a more explicit way to determine orientation.

\begin{lemma}\label{lem:sameorlocalconditions}
	Let $B$ be a quaternion algebra over $\QQ_p$ with Eichler order $\Ord$ of level $\mathfrak{M}$, let $D$ be a discriminant, and let $\phi:\qordD\rightarrow\Ord$ be an optimal embedding.
	\begin{enumerate}[label=(\roman*)]
		\item If $B$ is division, let $\mathfrak{p}$ be the maximal ideal of $\Ord$. Then the orientation of $\phi$ is determined by $\phi\left(\frac{p_D+\sqrt{D}}{2}\right)\pmod{\mathfrak{p}}$.
		\item If $B=\Mat(2, \QQ_p)$, $\Ord$ is the standard Eichler order of level $p^e$ with $e>0$, and $p\nmid D$, then the orientation of $\phi$ is determined by $\phi\left(\frac{p_D+\sqrt{D}}{2}\right)_{1,1}\pmod{p^e}$.
	\end{enumerate}
\end{lemma}
\begin{proof}
	If $B$ is division, then $\Ord/\mathfrak{p}\simeq\mathbb{F}_{p^2}$ is commutative. Thus, equivalent embeddings give the same value of $\phi\left(\frac{p_D+\sqrt{D}}{2}\right)\pmod{\mathfrak{p}}$. If $p\mid D$ we are done, and otherwise, note that $\overline{\phi}$ (defined by $\overline{\phi}(x):=\overline{\phi(x)}$) is an optimal embedding with 
	\[\overline{\phi}\left(\dfrac{p_D+\sqrt{D}}{2}\right)\not\equiv\phi\left(\dfrac{p_D+\sqrt{D}}{2}\right)\pmod{\mathfrak{p}},\]
	since this is equivalent to $\phi(\sqrt{D})\not\equiv 0\pmod{\mathfrak{p}}$. As there are two equivalence classes of optimal embeddings, it follows that the class is determined by $\phi\left(\frac{p_D+\sqrt{D}}{2}\right)\pmod{\mathfrak{p}}$.
	
	If $B=\Mat(2, \QQ_p)$, then a direct computation shows that $\phi\left(\frac{p_D+\sqrt{D}}{2}\right)_{1,1}\equiv\phi^u\left(\frac{p_D+\sqrt{D}}{2}\right)_{1,1}\pmod{p^e}$ for $u\in\Ord^1$ (see Equations \eqref{eqn:conjugateembor} and \eqref{eqn:conjugateembor2} for this computation). As in the previous case, $\overline{\phi}$ is an optimal embedding with
	\[\overline{\phi}\left(\dfrac{p_D+\sqrt{D}}{2}\right)\not\equiv\phi\left(\dfrac{p_D+\sqrt{D}}{2}\right)\pmod{p^e},\]
	since $p\nmid D$. As there are two equivalence classes of optimal embeddings, it follows that the class is determined by $\phi\left(\frac{p_D+\sqrt{D}}{2}\right)_{1,1}\pmod{p^e}$.
\end{proof}

For $p\mid\mathfrak{D}\mathfrak{M}\infty$, we can use the elements $\omega_p\in N_B^{\times}(\Ord)$ as described in Proposition \ref{prop:stab} to pass between orientations.

\begin{proposition}\label{prop:conjorient}
	Let $B$ be an indefinite quaternion algebra over $\QQ$ of discriminant $\mathfrak{D}$ with Eichler order $\Ord$ of level $\mathfrak{M}$, and let $\phi:\qordD\rightarrow\Ord$ be an optimal embedding. Then we have
	\begin{itemize}
		\item $o_{v}(\phi^{\omega_p})=o_{v}(\phi)$ for all places $v\neq p$;
		\item $o_{p}(\phi^{\omega_p})=-o_{p}(\phi)$ if $p\nmid\gcd(D,\mathfrak{M})$.
	\end{itemize}
	In other words, the optimal embedding $\phi^{\omega_p}$ only swaps orientation at $p$.
\end{proposition}
\begin{proof}
	If $v\mid\mathfrak{D}$, let $\mathfrak{v}$ be the maximal order of $\Ord_v$. Since $\Ord/\mathfrak{v}\simeq\mathbb{F}_{v^2}$ is commutative, the result follows for $p\neq v$ as $\nrd(\omega_p)\in\ZZ_v^{\times}$. If $p=v$, then we can assume that $p\nmid D$, as the result is trivial otherwise. By Proposition \ref{prop:countlocalfactors}, $\left(\frac{D}{p}\right)=-1$, whence we can write $B_p=\left(\frac{p, D}{\QQ_p}\right)$. It suffices to prove the proposition for $\omega_p=i$ and $\phi_p(\sqrt{D})=j$, and we indeed find that
	\[\phi_p^i\left(\dfrac{p_D+\sqrt{D}}{2}\right)=\dfrac{p_D+iji^{-1}}{2}=\dfrac{p_D-j}{2}\not\equiv \dfrac{p_D+j}{2}=\phi_p\left(\dfrac{p_D+\sqrt{D}}{2}\right)\pmod{\mathfrak{p}}.\]
	By Lemma \ref{lem:sameorlocalconditions}, the embeddings have opposite orientation.
	
	Next, take $v\mid\mathfrak{M}$, and assume that $\Ord_v$ is the standard Eichler order of level $v^e$. If $p\neq v$, then the computations in Equations \eqref{eqn:conjugateembor} and \eqref{eqn:conjugateembor2} still hold true, and so we are done by Lemma \ref{lem:sameorlocalconditions}. If $p=v$, then it suffices to take $\omega_p=\sm{0}{1}{-p^e}{0}$. If $\phi_v\left(\frac{p_D+\sqrt{D}}{2}\right)=\sm{a}{b}{p^ec}{p_D-a}$, then a direct computation shows that
	\[\phi_v^{\omega_p}=\lm{p_D-a}{-c}{-p^eb}{a},\]
	whence by Lemma \ref{lem:sameorlocalconditions}, the embeddings have the opposite orientation if and only if $a\not\equiv p_D-a\pmod{p^e}$. Assume otherwise, so that $2a-p_D\equiv 0\pmod{p^e}$. If $p=2$, then $D$ is odd, and this is not possible. If $p$ is odd, then by doubling the matrix expression for $\phi$, we see that $-(2a-p_D)^2\equiv D\pmod{p^e}$, hence this cannot be zero, as desired.
	
	Finally, if $v=\infty$, then this follows directly by definition and an explicit computation.
\end{proof}

If $\gcd(D,\mathfrak{M})=1$, then by successively conjugating an embedding by the elements $\omega_p$ for $p\mid\mathfrak{D}\mathfrak{M}\infty$, we can pass between all possible orientations. In particular, this implies that for all orientations $o$,
\[\vert\Emb_o(\Ord,D)\vert=h^+(D).\]
In fact, more is true: there is a simply transitive action of the narrow class group $\Cl^+(D)$ on $\Emb_o(\Ord,D)$, valid for all discriminants $D$ for which $\Emb(\Ord,D)$ is non-empty. See Section 4.5 of \cite{JRthe}, or the discussion below Definition 4.22 of \cite{DV20} for more details.

\section{Basic results on intersection numbers}\label{sec:basicresults}

With the background out of the way, we turn our focus to intersection numbers. Proposition 1.10 of \cite{JR21} gives nice descriptions of when root geodesics of hyperbolic matrices in $\SL(2,\RR)$ intersect. We state the relevant parts here (and change the expression for $\tan(\theta)$ into $\cos(\theta)$).

\begin{proposition}[Proposition 1.10 of \cite{JR21}]\label{prop:matintersect}
	Let $M_1,M_2\in \SL(2,\RR)$ be hyperbolic matrices with respective upper half plane root geodesics $\ell_1,\ell_2$, and let $Z_i=M_i-\frac{\Tr(M_i)}{2}\Id$ for $i=1,2$. Then
	\begin{enumerate}[label=(\roman*)]
		\item $\ell_1,\ell_2$ intersect transversely if and only if
		\[\det(M_1M_2-M_2M_1)>0.\]
		\item In all cases,
		\[\det(M_1M_2-M_2M_1)=\det(Z_1Z_2-Z_2Z_1)=4\det(Z_1Z_2)-(\Tr(Z_1Z_2))^2.\]
		\item If $\ell_1,\ell_2$ intersect transversely, then 
		\begin{enumerate}
			\item the intersection point is the fixed point of $Z_1Z_2$ that lies in $\uhp$.
			\item the intersection angle $\theta$ (measured counterclockwise from the tangent to $\ell_1$ to the tangent to $\ell_2$) satisfies 
			\[\cos(\theta)=\dfrac{\Tr(Z_1Z_2)}{2\sqrt{\det(Z_1Z_2)}}.\]
		\end{enumerate}
	\end{enumerate}
\end{proposition}

In particular, Proposition \ref{prop:optembint} is a corollary of this proposition.

\begin{proof}[Proof of Proposition \ref{prop:optembint}]
	Assume that $(\phi_1,\phi_2)$ are $x-$linked optimal embeddings of positive discriminants $D_1, D_2$. Let $M_i=\iota(\phi_i(\epsilon_{D_i}))$ for $i=1,2$, where the fundamental units can be written as $\epsilon_{D_i}=\frac{T_i+U_i\sqrt{D_i}}{2}$, with $(T_i,U_i)$ being the smallest positive integer solution to $T^2-D_iU^2=4$. In particular, $\ell_{\phi_i}=\ell_i$ for $i=1,2$. It follows that $Z_i=\frac{U_i}{2}\iota(\phi_i(\sqrt{D_i}))$, hence
	\[\det(Z_i)=\dfrac{-U_i^2D_i}{4},\qquad \Tr(Z_1Z_2)=\dfrac{U_1U_2x}{2}.\]
	Therefore
	\[4\det(Z_1Z_2)-(\Tr(Z_1Z_2))^2=\dfrac{U_1^2U_2^2}{4}(D_1D_2-x^2),\]
	and the root geodesics intersect transversely if and only if $x^2<D_1D_2$. This proves the first claim.
	
	Assume the root geodesics intersect transversely, and let $T=\phi_1(\sqrt{D_1})\phi_2(\sqrt{D_2})$; the intersection point is the upper half plane fixed point of $\iota(T)$. Since $T$ satisfies $T^2-2xT+D_1D_2=0$, $T$ acts as $x+\sqrt{x^2-D_1D_2}$. As $T\in(2\Ord+p_{D_1})(2\Ord+p_{D_2})\subset 2\Ord+p_{D_1D_2}$, $T$ corresponds to an embedding of $\qord{x^2-D_1D_2}$ into $\Ord$, which is part i. This also implies that $x\equiv D_1D_2\pmod{2}$.
	
	Finally, the angle of intersection satisfies
	\[\cos(\theta)=\dfrac{U_1U_2x/2}{2\sqrt{U_1^2U_2^2D_1D_2/16}}=\dfrac{x}{\sqrt{D_1D_2}},\]
	and the proof is finished.
\end{proof}

This implies that we can replace ``study intersections of $\ell_{\phi_1},\ell_{\phi_2}$'' by ``study $\Emb(\Ord,\phi_1,\phi_2,x)$ for $x^2<D_1D_2$.'' 

While the sets $\Emb(\Ord,\phi_1,\phi_2,x)$ can be computed in practice, it is a much harder task to access their theoretical properties. Instead, from now on we will focus on $\Emb(\Ord,D_1, D_2,x)$, for positive discriminants $D_1, D_2$, which captures all possible $x-$linking of optimal embeddings of discriminants $D_1, D_2$ into $\Ord$.

While we will eventually characterize and count $\Emb(\Ord,D_1,D_2,x)$, we can already prove a strong necessary condition for this set to be non-empty.

\begin{lemma}\label{lem:primesdivnorm}
	Let $v_1,v_2\in\Ord$. Then 
	\[\mathfrak{D}\mathfrak{M}\mid\nrd(v_1v_2-v_2v_1).\]
\end{lemma}
\begin{proof}
	Let $p\mid\mathfrak{D}\mathfrak{M}$, and consider completing $B$ at $p$. We can assume that the completion $\Ord_p$ is either the unique maximal order if $B_p$ is division, or the standard Eichler order of level $p^e$ otherwise. In the first case, let the unique maximal ideal of $\Ord_p$ be $\mathfrak{p}$, and then $\frac{\Ord_p}{\mathfrak{p}}\simeq \mathbb{F}_{p^2}$ is a field. Thus
	\[v_1v_2\equiv v_2v_1\pmod{\mathfrak{p}},\]
	which implies that $v_1v_2-v_2v_1\in\mathfrak{p}$, and so $p\mid\nrd(v_1v_2-v_2v_1)$.
	
	The second case follows from the fact that looking modulo $p^e$, we have upper triangular matrices. The diagonal of their product is unchanged when we swap the order of multiplication, and the result follows.
\end{proof}

\begin{corollary}\label{cor:finitequat}
	If $(\phi_1,\phi_2)$ are $x-$linked, then 
	\[\mathfrak{D}\mathfrak{M}\mid\dfrac{D_1D_2-x^2}{4}.\]
	In particular, for a fixed pair of discriminants $D_1,D_2$, there is a finite set of non-isomorphic pairs $(B, \Ord)$ of an indefinite quaternion algebra $B$ over $\QQ$ with Eichler order $\Ord$ for which there exist optimal embeddings of $D_1,D_2$ into $\Ord$ giving a non-zero unweighted intersection number.
\end{corollary}
\begin{proof}
	Let $v_i=\frac{p_{D_i}+\sqrt{D_i}}{2}$, and using Lemma \ref{lem:primesdivnorm} and a computation analogous to Proposition \ref{prop:matintersect}ii, we compute
	\begin{align*}
		\mathfrak{D}\mathfrak{M} & \mid \nrd(\phi_1(v_1)\phi_2(v_2)-\phi_2(v_2)\phi_1(v_1))\\
		& \quad = \dfrac{\nrd(\phi_1(\sqrt{D_1})\phi_2(\sqrt{D_2})-\phi_2(\sqrt{D_2})\phi_1(\sqrt{D_1}))}{16}=\dfrac{D_1D_2-x^2}{4}.
	\end{align*}
	Intersections come from the finite set of $x$ for which $x^2<D_1D_2$, and this calculation shows that for each such $x$ there are finitely many pairs $(\mathfrak{D},\mathfrak{M})$ that satisfy the divisibility condition (in Theorem \ref{thm:onequatalgxlink} we will show that $\mathfrak{D}$ is in fact uniquely determined from $D_1,D_2,x$). Therefore, there are finitely many Eichler orders for which there exist intersections of optimal embeddings of discriminants $D_1,D_2$.
\end{proof}

\section{Existence of x-linked pairs}\label{sec:xlink}

Rather than study the set $\Emb(\Ord, D_1, D_2, x)$ directly, we invert the setup. That is, we start with a pair of $x-$linked embeddings into $B$, and consider the possible Eichler orders which admit these (optimal) embeddings. We study this problem locally, and show how to lift the local results to global results. In this section, we we start this process by studying which quaternion algebras admit $x-$linked embeddings. 

\subsection{Simultaneous conjugation}
The fact that we are only allowing conjugation by elements of $\Ord^1$ and not all of $B^{\times}$ is crucial to $x-$linking.

\begin{lemma}\label{lem:pairsconjugate}
	Let $B$ be a quaternion algebra over a field $F$, and let $(x_1,x_2)$ and $(y_1,y_2)$ be pairs of elements of $B^{\times}$ for which:
	\begin{itemize}
		\item $x_i,y_i\notin F$ for $i=1,2$;
		\item $x_i$ and $y_i$ have the same irreducible minimal polynomial over $F$ for $i=1,2$;
		\item $x_1x_2$ and $y_1y_2$ have the same minimal polynomial over $F$.
	\end{itemize}
	Then the pairs are simultaneously conjugate over $B^{\times}$, i.e. there exists an $r\in B^{\times}$ for which $rx_1r^{-1}=x_2$ and $ry_1r^{-1}=y_2$.
\end{lemma}
\begin{proof}
	The $F-$algebras $F[x_1,x_2]$ and $F[y_1,y_2]$ are $F-$subalgebras of $B$ of (equal) dimension $2$ or $4$. If they have dimension $4$, then they are equal to $B$, and are thus simple. Otherwise, they are equal to $F[x_1]$ and $F[y_1]$, which are again simple algebras since the minimal polynomials were irreducible.
	
	Consider the map $\theta:F[x_1,x_2]\rightarrow F[y_1,y_2]$ defined by $\theta(x_i)=y_i$ for $i=1,2$. The equality of the minimal polynomials of $x_i,y_i$ and $x_1x_2,y_1y_2$ implies that the map is indeed a well defined isomorphism. By the Skolem-Noether theorem, this map is inner in $B$ (Corollary 7.7.2 of \cite{JV21}), and this implies the result.
\end{proof}

Applying Lemma \ref{lem:pairsconjugate} to optimal embeddings produces the following corollary.

\begin{corollary}\label{cor:allpairsconj}
	Let $B$ be a quaternion algebra over $F=\QQ$ or $\QQ_p$, and let $(\phi_1,\phi_2),(\phi_1',\phi_2')$ be pairs of $x-$linked embeddings from $\qord{D_1},\qord{D_2}$ respectively into $B$. Then $V=\{v\in B: v\phi_n=\phi_n'v\text{ for $n=1,2$}\}$ is a $1$-dimensional $F$-vector space, generated by an element of $B$ with non-zero norm. In particular, the pairs of embeddings are simultaneously conjugate over $B^{\times}$.
\end{corollary}
\begin{proof}
	Let $V_n=\{v\in B: v\phi_n=\phi_n'v\}$ for $n=1,2$; by Corollary \ref{cor:eltconj}, this is a two dimensional $F$-vector space. Furthermore, we have $V_n=r_n(F+\phi_n(\sqrt{D_n})F)$ for $n=1,2$ for some $r_1,r_2\in B^{\times}$. We claim that $V_1$ and $V_2$ are distinct: otherwise, right multiplication by $\phi_1(\sqrt{D_1})$ on $V_1$ remains in $V_1$, hence it is true for $V_2$ as well. This implies that $\phi_1(\sqrt{D_1})\in F+\phi_2(\sqrt{D_2})F$, and therefore $\phi_1(\sqrt{D_1})$ is a scalar multiple of $\phi_2(\sqrt{D_2})$ (by taking traces). Writing $\phi_1(\sqrt{D_1})=f\phi_2(\sqrt{D_2})$ for $f\in F^{\times}$, squaring gives us $D_1=f^2D_2$ and $x=\frac{1}{2}\trd\left(\phi_1(\sqrt{D_1})\phi_2(\sqrt{D_2})\right)=fD_2$. Thus $x^2=f^2D_2^2=D_1D_2$, which is a contradiction by definition of $x-$linkage.
	
	Since $V=V_1\cap V_2$, $V$ has dimension $0$ or $1$ as $V_1,V_2$ are distinct. We apply Lemma \ref{lem:pairsconjugate} to the images of $\sqrt{D_1},\sqrt{D_2}$ under $(\phi_1,\phi_2)$ and $(\phi_1',\phi_2')$ respectively. The minimal polynomials satisfy the requirements, whence the lemma implies that $V$ has an invertible element. Thus $V$ has dimension $1$, as desired.
\end{proof}

\subsection{Orders containing x-linked pairs}

Given a pair of embeddings $\phi_i:\qord{D_i}\rightarrow B$ ($i=1,2$), there does not need to be an order that contains the images of both $\qord{D_i}$. The following definition and lemma describe when there is such an order.

\begin{definition}
	Let $(D_1, D_2, x)$ be a triple of integers. We call the triple \textit{admissible} if the following hold:
	\begin{itemize}
		\item $D_1$ and $D_2$ are positive discriminants;
		\item $x\equiv D_1D_2\pmod{2}$ and $x^2\neq D_1D_2$.
	\end{itemize}
\end{definition}

A consequence of the following lemma is that there exists an order containing given $x-$linked embeddings of discriminants $D_1, D_2$ if $(D_1, D_2, x)$ is admissible.

\begin{lemma}\label{lem:genorder}
	Let $F=\QQ$ or $\QQ_p$, and let $B$ be a quaternion algebra over $F$. Let $\phi_i:\qord{D_i}\rightarrow B$ be embeddings of the orders of discriminants $D_1,D_2$ into $B$, and take $v_i=\phi_i\left(\frac{p_{D_i}+\sqrt{D_i}}{2}\right)$ for $i=1,2$. Assume that $x=\frac{1}{2}\trd(\phi_1(\sqrt{D_1})\phi_2(\sqrt{D_2}))\in p_{D_1D_2}+2\mathcal{O}_F$ and $x^2\neq D_1D_2$. Then
	\[\Ord_{\phi_1,\phi_2}:=\langle 1,v_1,v_2,v_1v_2\rangle_{\mathcal{O}_F}\]
	is an order of $B$, necessarily the smallest order of $B$ for which both $\phi_1$ and $\phi_2$ embed into. Furthermore, 
	\[\discrd(\Ord_{\phi_1,\phi_2})=\frac{D_1D_2-x^2}{4}.\]
\end{lemma}
\begin{proof}
	For ease of notation write $\Ord=\Ord_{\phi_1,\phi_2}$. First, 
	\[\trd(v_1v_2)=\dfrac{p_{D_1}p_{D_2}+x}{2}=p_{D_1D_2}+\dfrac{x-p_{D_1D_2}}{2}\in\mathcal{O}_F,\]
	and $\nrd(v_1v_2)=\nrd(v_1)\nrd(v_2)\in\mathcal{O}_F$, whence $v_1v_2$ is integral. We will demonstrate that $v_2v_1\in\Ord$, and the rest of the equations to prove that $\Ord_{\phi_1,\phi_2}$ is closed under multiplication can be deduced from this and the minimal polynomials for $v_1,v_2$. We compute
	\begin{align*}
		v_1v_2+v_2v_1 = & \dfrac{p_{D_1D_2}+p_{D_1}\phi_2(\sqrt{D_2})+p_{D_2}\phi_1(\sqrt{D_1})}{2}\\
		& + \dfrac{\phi_1(\sqrt{D_1})\phi_2(\sqrt{D_2})+\phi_2(\sqrt{D_2})\phi_1(\sqrt{D_1})}{4}\\
		= & p_{D_1}v_2+p_{D_2}v_1+\dfrac{x-p_{D_1D_2}}{2},
	\end{align*}
	whence $v_2v_1$ lies in $\Ord$, as claimed. 
	
	The fact that $\Ord$ is an order will follow from computing its reduced discriminant, and seeing that it is non-zero. To ease our calculations, write
	\[\left(\begin{matrix}1\\ \phi_1(\sqrt{D_1})\\ \phi_2(\sqrt{D_2}) \\ \phi_1(\sqrt{D_1})\phi_2(\sqrt{D_2})\end{matrix}\right)
	=\left(\begin{matrix} 1 & 0 & 0 & 0 \\ -p_{D_1} & 2 & 0 & 0 \\ -p_{D_2} & 0 & 2 & 0 \\ p_{D_1}p_{D_2} & -2p_{D_2} & -2p_{D_1} & 4 \end{matrix}\right)
	\left(\begin{matrix} 1\\ v_1\\ v_2 \\ v_1v_2\end{matrix}\right),
	\]
	and we have the equation 
	\[d(1,\phi_1(\sqrt{D_1}), \phi_2(\sqrt{D_2}),\phi_1(\sqrt{D_1})\phi_2(\sqrt{D_2}))=\det(M)^2 d(1,v_1,v_2,v_1v_2),\]
	where $M$ is the transition matrix above. We compute $\det(M)=16$ and 
	\begin{align*}
		d(1,\phi_1(\sqrt{D_1}), \phi_2(\sqrt{D_2}),\phi_1(\sqrt{D_1})\phi_2(\sqrt{D_2})) & = \det\left(\begin{matrix} 2 & 0 & 0 & 2x\\ 0 & 2D_1 & 2x & 0\\ 0 & 2x & 2D_2 & 0\\ 2x & 0 & 0 & 4x^2-2D_1D_2\end{matrix}\right)\\
		& = -16(D_1D_2-x^2)^2.
	\end{align*}
	Since $\discrd(\Ord)^2=-d(1,v_1,v_2,v_1v_2)$, the reduced discriminant is as claimed (and is non-zero by the assumption of $x^2\neq D_1D_2$).
	
	It is immediate that $\Ord$ is the smallest order for which $\phi_1,\phi_2$ embed into, as such an order must contain $\{1,v_1,v_2\}$, and $\Ord$ is generated as an $\mathcal{O}_F$ algebra by these elements.
\end{proof}

Lemma \ref{lem:genorder} has some historical connections. The proof of Theorem 2' in \cite{Kaneko89} details a similar computation in a definite quaternion algebra. Furthermore, as noted by Gross, this definite computation leads to a simple argument that a prime $p$ dividing $\Nm(j(\tau_1)-j(\tau_2))$ must satisfy $p\mid\frac{D_1D_2-x^2}{4}$ for $x^2<D_1D_2$ (see Theorem 65 and Proposition 66 of \cite{franCMnotes} for the full argument).

Our first application of Lemma \ref{lem:genorder} is to show that $x-$linked pairs of embeddings can be detected locally.

\begin{lemma}\label{lem:localglobalxlinking}
\begin{sloppypar}
	Let $B$ be an indefinite quaternion algebra over $\QQ$, let $\Ord$ be an Eichler order in $B$, and let $(D_1, D_2, x)$ be an admissible triple. Then the set $\Emb(\Ord,D_1,D_2,x)$ is non-empty if and only if $\Emb(\Ord_p,D_1,D_2,x)$ is non-empty for all finite primes $p$.
\end{sloppypar}
\end{lemma}
\begin{proof}
	If such a pair $(\phi_1,\phi_2)\in\Emb(\Ord,D_1,D_2,x)$ exists, then the corresponding maps to the completions gives elements of $\Emb(\Ord_p,D_1,D_2,x)$ for all $p$. 
	
	To prove the opposite direction, assume that $(\alpha_p,\beta_p)\in\Emb(\Ord_p,D_1,D_2,x)$ for all $p$. A consequence of Proposition \ref{prop:countembD} is that there exists an embedding $\phi_1$ of $\mathcal{O}_{D_1}$ into $B$. By Corollary \ref{cor:optembgoestoi}, we can assign coordinates so that $\phi_1(\sqrt{D_1})=i$. In this case, we are considering the existence of a map $\phi_2$ such that $\phi_2(\sqrt{D_2})=fi+gj+hk$, where
	\[\nrd(fi+gj+hk)=-D_2\qquad\text{and}\qquad 2x=\trd(i(fi+gj+hk))=2fD_1.\]
	With the substitution of $f=\frac{x}{D_1}$, the equation $\nrd\left(\frac{x}{D_1}i+gj+hk\right)+D_2=0$ is a quadratic form in $g,h$. This will have a solution in $\RR$ since $B$ is indefinite, and it will have a solution in $\QQ_p$ for all $p$ since $\Emb(\Ord_p,D_1,D_2,x)$ is non-empty. By Hasse's principle, it has a solution over $\QQ$; let the corresponding map be $\phi_2$.
	
	Following Lemma \ref{lem:genorder}, let $\Ord'=\Ord_{\phi_1,\phi_2}$ be the smallest order for which $\phi_1,\phi_2$ embed into. By Corollary \ref{cor:allpairsconj}, for all finite primes $p$ there exists an $r_p\in B_p^{\times}$ for which $r_p(\alpha_p,\beta_p)r_p^{-1}=(\phi_{1,p},\phi_{2,p})$. By the definition of $\Ord'$, it follows that $\Ord_p'\subseteq r_p\Ord_p r_p^{-1}$. For all primes $p$,
	\begin{itemize}
		\item let $s_p=r_p$ if $\Ord_p'\neq\Ord_p$ or $p\mid D_1D_2$;
		\item let $s_p=1$ otherwise.
	\end{itemize}
	Consider the sequence of local orders $\{s_p\Ord_p s_p^{-1}\}_p$. Since $\Ord_p'=\Ord_p$ holds for all but finitely many primes, by Theorem \ref{thm:JVlocalglobal} there exists an order $\Ord''$ of $B$ which completes to $s_p\Ord_p s_p^{-1}$ for all primes $p$. In particular, we note that $\Ord''$ is an Eichler order of level $\mathfrak{M}$, and $\phi_1,\phi_2$ give embeddings into $\Ord''$. When $p\mid D_1D_2$ the local embeddings are optimal since $(\alpha_p,\beta_p)$ were optimal, hence $\phi_1,\phi_2$ are optimal embeddings into $\Ord''$. Since all Eichler orders of the same level are conjugate, let $r\Ord''r^{-1}=\Ord$, and then $[r(\phi_1,\phi_2)r^{-1}]\in\Emb(\Ord,D_1,D_2,x)$, as required.
\end{proof}

In particular, the non-emptyness of $\Emb(\Ord,D_1,D_2,x)$ can be studied locally.

\subsection{Local x-linking}
While we were concerned with orders in Lemma \ref{lem:localglobalxlinking}, we will drop this for now and instead consider embeddings into the entire quaternion algebra.

\begin{definition}
	Let $(D_1, D_2, x)$ be an admissible triple, and define $\Emb(B,D_1,D_2,x)$ to be the set of all pairs $(\phi_1,\phi_2)$ of $x-$linked embeddings of discriminants $D_1,D_2$ into $B$.
\end{definition}

Note that Lemma \ref{lem:localglobalxlinking} also applies to the sets $\Emb(B,D_1,D_2,x)$ and $\Emb(B_p,D_1,D_2,x)$. Our next goal is to determine when $\Emb(B_p,D_1,D_2,x)$ is non-empty. Before getting into these local computations, we require a lemma about the solutions to Pell's equation over $\ZZ_p$.

\begin{lemma}\label{lem:pellmodp}
	Let $p$ be a prime, let $A$ be a non-zero integer, and let $D$ be a positive discriminant coprime to $p$. Then the equation 
	\begin{equation}\label{eqn:pell}
		X^2-DY^2=A
	\end{equation}
	has a solution $(X,Y)\in\ZZ_p^2$ if and only if one of the following conditions hold:
	\begin{itemize}
		\item $\left(\frac{D}{p}\right)=1$, and if $p=2$ we additionally have $v_2(A)\neq 1$;
		\item $\left(\frac{D}{p}\right)=-1$ and $v_p(A)$ is even.
	\end{itemize}
\end{lemma}
\begin{proof}
	If $\left(\frac{D}{p}\right)=1$, then $\sqrt{D}\in\ZZ_p$ (noting that if $p=2$ then $D\equiv 1\pmod{8}$ as it is a discriminant). By factoring $(X-\sqrt{D}Y)(X+\sqrt{D}Y)=A=uv$, this will always have a solution if $p$ is odd. If $p=2$, then we require $u$ and $v$ to have opposite parity, which gives $v_2(A)\neq 1$.
	
	Otherwise, $\QQ_p(\sqrt{D})$ is the unramified degree $2$ extension of $\QQ_p$, and $X^2-DY^2$ is the norm form from $\ZZ_p[\sqrt{D}]$ to $\ZZ_p$. The result follows, e.g. by Chapter 2, Section 4 of \cite{Lang94}.
\end{proof}

We start the local calculations by considering the division algebra case. Recall the Hilbert symbol $(a,b)_p$, which has an alternate characterization via Hilbert's criterion: $(a,b)_p=1$ if and only if $ax^2+by^2=1$ has solutions with $x,y\in\mathbb{Q}_p$ (see Section 12.4 of \cite{JV21}).

\begin{lemma}\label{lem:localdivisionxlinking}
\begin{sloppypar}
	Let $(D_1, D_2, x)$ be an admissible triple, and let $B$ be the division algebra over $\QQ_p$.Then $\Emb(B,D_1,D_2,x)$ is non-empty if and only if
	\[(D_1,x^2-D_1D_2)_p=-1.\]
	If $p\nmid D_1$, this is equivalent to
	\[\left(\frac{D_1}{p}\right)=-1 \text{ and } v_p\left(\dfrac{D_1D_2-x^2}{4}\right)\text{ is odd.}\]
\end{sloppypar}
\end{lemma} 
\begin{proof}
	If there does not exist an embedding of $\qord{D_1}$ into $B$, then the same is true for $\qord{D_1^{\fund}}$. By Proposition \ref{prop:countlocalfactors}ii, $\left(\frac{D_1^{\fund}}{p}\right)=1$, and therefore by Hilbert's criterion, $(D_1,N)_p=1$ for all $N\neq 0$. In particular, $(D_1,x^2-D_1D_2)_p\neq -1$, as desired.
	
	Otherwise, by Corollary \ref{cor:optembgoestoi} we can write $B=\left(\frac{D_1,e}{\QQ_p}\right)$ for some non-zero $e\in\ZZ_p$, where $\phi_1(\sqrt{D_1})=i$ and $(D_1,e)_p=-1$. Writing $\phi_2(\sqrt{D_2})=fi+gj+hk$, it suffices to solve the equations
	\[D_1f^2+eg^2-D_1eh^2=D_2,\qquad x=fD_1.\]
	Therefore $f=\frac{x}{D_1}$, and the first equation rearranges to 
	\begin{equation}\label{eqn:localdivisionxlink}
		g^2-D_1h^2=\dfrac{D_1D_2-x^2}{eD_1}.
	\end{equation}
	If this has a solution with $h=h_1$, then by Hensel's lemma there will be a solution with $h=h_1+p^k$ for large enough $k$. In particular, they correspond to distinct $g$'s, so we can solve the equation with the assumption that $g\neq 0$. Equation \eqref{eqn:localdivisionxlink} then rearranges to
	\[D_1(h/g)^2+\dfrac{D_1D_2-x^2}{eD_1}(1/g)^2=1,\]
	which is in the format of Hilbert's criterion. The properties of the Hilbert symbol imply that
	\[1=\left(D_1,\frac{D_1D_2-x^2}{eD_1}\right)_p=(D_1,(x^2-D_1D_2)e)_p=-(D_1,x^2-D_1D_2)_p,\]
	from which the first result follows.
	
	If $p\nmid D_1$, then $\left(\frac{D_1}{p}\right)=-1$, and we claim that $v_p(e)$ is odd. If $p$ is odd this follows immediately, since $(a,b)_p=1$ if $p\nmid ab$. If $p=2$, then $D_1\equiv 5\pmod{8}$, and this follows by computing $(D_1,b)_2$ for all $b\in \QQ_2^{\times}/Q_2^{\times 2}$, and seeing that $(D_1,b)_2=1$ if $v_2(b)$ is even (see Table 12.4.16 of \cite{JV21} for this computation).
	
	Scaling Equation \eqref{eqn:localdivisionxlink} by powers of $p$, it is equivalent to solve
	\begin{equation}\label{eqn:localdivisionsuff}
		g^2-D_1h^2=\dfrac{D_1D_2-x^2}{eD_1}p^{2r},
	\end{equation}
	for $r\geq 0$ and $g,h\in\ZZ_p$. Lemma \ref{lem:pellmodp} implies that Equation \eqref{eqn:localdivisionsuff} has a solution if and only if \[v_p\left(\dfrac{D_1D_2-x^2}{eD_1}p^{2r}\right)\]
	is even, which is equivalent to our condition.
\end{proof}

Now we consider non-division algebras.

\begin{lemma}\label{lem:localnondivisionxlinking}
	Let $(D_1, D_2, x)$ be an admissible triple. Then $\Emb(\Mat(2,\QQ_p),D_1,D_2,x)$ is non-empty if and only if 
	\[(D_1,x^2-D_1D_2)_p=1.\]
	If $p\nmid D_1$ this is equivalent to either
	\[\left(\dfrac{D_1}{p}\right)=1\qquad\text{ or }\qquad \left(\frac{D_1}{p}\right)=-1\text{ and } v_p\left(\frac{D_1D_2-x^2}{4}\right)\text{ is even.}\]
\end{lemma}
\begin{proof}
	Since embeddings of a fixed discriminant are all conjugate over $B^{\times}$, we can fix the first embedding to be $\phi_1(\sqrt{D_1})=\sm{0}{D_1}{1}{0}$, and write $\phi_2(\sqrt{D_2})=\sm{e}{f}{g}{-e}\in\Mat(2,\QQ_p)$. We will have a solution if and only if
	\[e^2+fg=D_2,\qquad D_1g+f=2x.\]
	This implies that $f=2x-D_1g$, and plugging this into the first equation and rearranging gives
	\[e^2-D_1\left(g-\dfrac{x}{D_1}\right)^2=\dfrac{D_1D_2-x^2}{D_1}.\]
	Let $X=e$ and $Y=g-\frac{x}{D_1}$, and then the equation is
	\[X^2-D_1Y^2=\dfrac{D_1D_2-x^2}{D_1}.\]
	The rest of the proof is analogous to Lemma \ref{lem:localdivisionxlinking}, where Lemma \ref{lem:pellmodp} completes the characterization of the solubility when $p\nmid D_1$.
\end{proof}

Lemmas \ref{lem:localdivisionxlinking} and \ref{lem:localnondivisionxlinking} immediately imply the following corollary

\begin{corollary}\label{cor:onlyonexlinked}
	Let $(D_1, D_2, x)$ be an admissible triple, and let $B$ be the division algebra over $\QQ_p$. Then exactly one of $B$ and $\Mat(2,\QQ_p)$ admits $x-$linked embeddings of $\qord{D_1},\qord{D_2}$, and which one is determined by if $(D_1,x^2-D_1D_2)_p$ is $-1$ or $1$, respectively.
\end{corollary}

\begin{remark}
	The first half of Lemmas \ref{lem:localdivisionxlinking}, \ref{lem:localnondivisionxlinking} and Corollary \ref{cor:onlyonexlinked} still holds when $p=\infty$, where $\QQ_{\infty}=\RR$.
\end{remark}

\subsection{Global x-linking}

Fix an admissible triple $(D_1, D_2, x)$. Corollary \ref{cor:onlyonexlinked} combined with Lemma \ref{lem:localglobalxlinking} implies that there is precisely one quaternion algebra $B$ over $\QQ$ for which there exist embeddings $\phi_i$ of $\qord{D_i}$ into $B$ that are $x-$linked, and it can be given by
\[\left(\dfrac{D_1,x^2-D_1D_2}{\QQ}\right).\]
We describe the ramification of this quaternion algebra by using a generalization of the $\epsilon$ function (Definition \ref{def:epsilonoriginal}).

\begin{definition}\label{def:epsilon}
	Let $D_1,D_2$ be discriminants, and let $p$ be any prime such that 
	\[p\nmid\gcd(D_1^{\fund},D_2^{\fund})\qquad\text{and}\qquad\left(\frac{D_1^{\fund}D_2^{\fund}}{p}\right)\neq -1.\]
	Define
	\[
	\epsilon(p):=\begin{cases} \left(\dfrac{D_1^{\fund}}{p}\right) & \text{ if $p$ and $D_1^{\fund}$ are coprime;}\\\\
		\left(\dfrac{D_2^{\fund}}{p}\right) & \text{ if $p$ and $D_2^{\fund}$ are coprime.}
	\end{cases}
	\]
\end{definition}

\begin{theorem}\label{thm:onequatalgxlink}
	Let $(D_1, D_2, x)$ be an admissible triple. Then the only quaternion algebra over $\QQ$ that admits $x-$linked embeddings from $\qord{D_1},\qord{D_2}$ is 
	\[B=\left(\dfrac{D_1,x^2-D_1D_2}{\QQ}\right).\]
	Furthermore, let $N=\gcd(D_1^{\fund},D_2^{\fund})$, and factorize
	\[\dfrac{D_1D_2-x^2}{4}=\pm N'\prod_{i=1}^{r}p_i^{2e_i+1}\prod_{i=1}^s q_i^{2f_i}\prod_{i=1}^t w_i^{g_i},\]
	where $N'$ is minimal so that $\frac{D_1D_2-x^2}{4N'}$ is coprime to $N$, $p_i$ are the primes for which $\epsilon(p_i)=-1$ that appear to an odd power, $q_i$ are the primes for which $\epsilon(q_i)=-1$ that appear to an even power, and $w_i$ are the primes for which $\epsilon(w_i)=1$. Then $B$ is ramified at
	\[\{p_1,p_2,\ldots,p_r\}\cup\{p:p\mid N', (D_1,x^2-D_1D_2)_p=-1\}.\]
\end{theorem}
\begin{proof}
	It suffices to compute $(D_1,x^2-D_1D_2)_p$ for $p\mid\frac{D_1D_2-x^2}{4}$ satisfying $p\nmid N'$. If $p\nmid D_1$, Lemmas \ref{lem:localdivisionxlinking} and \ref{lem:localnondivisionxlinking} imply that the Hilbert symbol is $-1$ if and only $\epsilon(p)=-1$ and $v_p\left(\frac{D_1D_2-x^2}{4}\right)$ is odd, i.e. $p=p_i$ for some $i$. Since
	\[(D_1,x^2-D_1D_2)_p=(D_2,x^2-D_1D_2)_p,\]
	the same holds for $p\nmid D_2$. As we assume that $p\nmid N'$, the final case is (without loss of generality) $p\nmid D_1^{\fund}$ and $p\mid D_1,D_2$. By Lemma \ref{lem:nonfundamentalxlink}, we can replace $(D_1,D_2,x)$ by $(D_1/p^2,D_2,x/p)$, and repeat.
\end{proof}

In particular, if $\gcd(D_1^{\fund},D_2^{\fund})=1$, then $B$ is ramified at exactly $\{p_1,p_2,\ldots,p_r\}$.

\begin{remark}
	The value of $(D_1, x^2-D_1D_2)_p$ for $p\mid N'$ is full of technical casework, and there is little benefit in listing the cases out.
\end{remark}

\begin{remark}\label{rem:wlogD1D2xlink}
	To work with an explicit $x-$linked pair, take $B=\left(\frac{D_1,x^2-D_1D_2}{\QQ}\right)$, and define
	\[\phi_1(\sqrt{D_1})=i, \qquad \phi_2(\sqrt{D_2})= \dfrac{xi+k}{D_1}.\]
	This pair is $x-$linked and corresponds to $\phi_1\times\phi_2(\sqrt{x^2-D_1D_2})=j$.
\end{remark}

\section{Counting Eichler orders containing x-linked pairs}\label{sec:countEO}
Thanks to Theorem \ref{thm:onequatalgxlink}, we have a good description of quaternion algebras that exhibit $x-$linking. We now turn our focus to describing Eichler orders that admit $x-$linked pairs, i.e. Eichler superorders of $\Ord_{\phi_1,\phi_2}$. Once again, it suffices to do this locally.

\subsection{Local Eichler orders containing x-linked pairs}

Up until now, we have mostly worked in full generality. However, as evidenced by the end of Theorem \ref{thm:onequatalgxlink}, this generality can (and will) start to make results rather unwieldy. As such, we would like to find a middle ground between a pleasant exposition and full generality. The following definition is our choice for such a middle ground.

\begin{definition}\label{def:nice}
	Given an an admissible triple $(D_1, D_2, x)$, we call it \textit{nice} if
	\[\gcd\left(D_1,D_2,D_1D_2-x^2\right)=1.\]
	Note that a nice triple has at least one of $D_1, D_2$ being odd.
	
	More generally, if $p$ is a prime, we call $p$ \textit{nice} (with respect to $(D_1,D_2,x)$) if
	\[p\nmid\gcd\left(D_1,D_2,D_1D_2-x^2\right).\]
\end{definition}

From now on, we will mostly be working with nice triples/nice primes.

In order to determine if an order is Eichler or not, we consider the Eichler symbol (see Section 24.3 of \cite{JV21}). Working in $B=\left(\frac{a,b}{\QQ_p}\right)$, for $\alpha\in B$, define
\[\Delta(\alpha)=\trd(\alpha)^2-4\nrd(\alpha)=4(af^2+bg^2-abh^2),\]
where $\alpha=e+fi+gj+hk$. For an order $\Ord$ of $B$, define $(\Ord,p)$ to be the set of values that $\left(\frac{\Delta(\alpha)}{p}\right)$ takes as $\alpha$ ranges over $\Ord$, where $\left(\frac{\cdot}{p}\right)$ is the Kronecker symbol.

\begin{lemma}\label{lem:Eichlercharacterization}
	The set $(\Ord,p)$ determines the possible Eichler superorders of $\Ord$ as follows:
	\begin{itemize}
		\item The order $\Ord$ is Eichler and non-maximal if and only if $(\Ord,p)=\{0,1\}$ (i.e. $\Ord$ is ``residually split'').
		\item If $-1\in(\Ord,p)$, then $\Ord$ is contained in precisely one maximal order.
	\end{itemize}
\end{lemma}
\begin{proof}
	The first point is a direct consequence of Lemma 24.3.6 of \cite{JV21}. For the second point, if $\Ord'$ is a superorder of $\Ord$, then $( \Ord,p)\subseteq(\Ord',p)$. In particular, no superset has $(\Ord,p)=\{0,1\}$, whence $\Ord$ is not contained in a non-maximal Eichler order. If $\Ord$ were contained in two maximal orders, it would be contained in their intersection, a non-maximal Eichler order, contradiction.
\end{proof}

Lemma \ref{lem:Eichlercharacterization} allows us to compute the Eichler orders containing $\Ord_{\phi_1,\phi_2}$.

\begin{lemma}\label{lem:Eichlernessoforder}
	Let $(D_1,D_2,x)$ be admissible, and let $\phi_1,\phi_2$ be $x-$linked embeddings of discriminants $D_1,D_2$ into $B$, a quaternion algebra over $\QQ_p$, where $p$ is \nice. Let $\Ord=\Ord_{\phi_1,\phi_2}$, and then:
	\begin{enumerate}[label=(\roman*)]
		\item If $p\nmid\frac{D_1D_2-x^2}{4}$, then $\Ord$ is maximal;
		\item If $\epsilon(p)=-1$, then $\Ord$ is contained in a unique maximal order;
		\item If $\epsilon(p)=1$, then $\Ord$ is Eichler.
	\end{enumerate}
\end{lemma}
\begin{proof}
	By Lemma \ref{lem:genorder}, the reduced discriminant of $\Ord$ is $\frac{D_1D_2-x^2}{4}$. Thus if $p\nmid\frac{D_1D_2-x^2}{4}$, $\Ord$ is maximal.
	
	Now, assume that $p\mid\frac{D_1D_2-x^2}{4}$, which implies that $p^{1+2v_2(p)}\mid x^2-D_1D_2$. As $p$ is nice, it follows that $p\nmid\gcd(D_1,D_2)$, so without loss of generality assume that $p\nmid D_1$. Take $B,\phi_i$ as in Remark \ref{rem:wlogD1D2xlink}, and then by Lemma \ref{lem:genorder}, a general element of $\Ord$ is of the form 
	\begin{align*}
		\alpha=A_0+A\dfrac{p_{D_1}+i}{2}+ & B\dfrac{p_{D_2}+(xi+k)/D_1}{2}+C\dfrac{j}{2}\\
		= & \left(A_0+A\dfrac{p_{D_1}}{2}+B\dfrac{p_{D_2}}{2}\right)+\left(\dfrac{A}{2}+\dfrac{Bx}{2D_1}\right)i+\dfrac{C}{2}j+\dfrac{B}{2D_1}k,
	\end{align*}
	for $A_0,A,B,C\in\ZZ_p$. Therefore
	\[\Delta(\alpha)= D_1\left(A+\frac{Bx}{D_1}\right)^2+(x^2-D_1D_2)C^2-D_1(x^2-D_1D_2)\frac{B^2}{D_1^2},\]
	whence
	\[\Delta(\alpha)\equiv D_1\left(A+\frac{Bx}{D_1}\right)^2\pmod{p^{1+2v_2(p)}}.\]
	Thus $(\Ord,p)=\{0,\epsilon(p)\}$, which by Lemma \ref{lem:Eichlercharacterization} completes the second and third points.
\end{proof}

Lemma \ref{lem:Eichlernessoforder} implies that locally, there is a minimal Eichler order containing $\Ord_{\phi_1,\phi_2,p}$, which is either the order itself, or the unique maximal order it is contained within. Therefore the result is true globally, and we make this a definition.

\begin{definition}
	Let $\phi_1,\phi_2$ be $x-$linked embeddings of discriminants $D_1,D_2$ into $B$, an indefinite quaternion algebra over $\QQ$ or $\QQ_p$, where $(D_1,D_2,x)$ is \nice. Then there exists a minimal Eichler order containing $\Ord_{\phi_1,\phi_2}$, denoted $\Ord_{\phi_1,\phi_2}^{\Eich}$.
\end{definition}

Since we are concerned with the optimality of embeddings, we need to determine which orders containing $\Ord_{\phi_1,\phi_2}^{\Eich}$ admit $\phi_1,\phi_2$ as optimal embeddings.

\begin{lemma}\label{lem:nonfundamentalxlink}
	Let $(D_1, D_2, x)$ be admissible, and let $\phi_1,\phi_2$ be $x-$linked embeddings of $D_1,D_2$ into $B$, an indefinite quaternion algebra over $\QQ$. Let $p$ be a prime for which $p\mid\frac{D_1}{D_1^{\fund}}$, and let $\phi_1'$ be the corresponding embedding of $\qord{D_1/p^2}$ into $B$ that agrees with $\phi_1$ on $\qord{D_1}$. Then $(\phi_1',\phi_2)$ are $\frac{x}{p}-$linked embeddings into $B$ if and only if $p\mid\frac{D_1D_2-x^2}{4}$.
\end{lemma}
\begin{proof}
	Since
	\[\dfrac{1}{2}\trd\left(\phi_1'\left(\sqrt{D_1/p^2}\right)\phi_2\left(\sqrt{D_2}\right)\right)=\dfrac{x}{p},\]
	$(\phi_1',\phi_2)$ are $\frac{x}{p}-$linked embeddings if and only if $\frac{x}{p}$ is an integer congruent to $\frac{D_1D_2}{p^2}$ modulo $2$.
	
	If this is the case, then by Lemma \ref{lem:genorder} the reduced discriminant of $\Ord_{\phi_1',\phi_2}$ is $\frac{D_1D_2-x^2}{4p^2}$, which implies that $p\mid\frac{D_1D_2-x^2}{4}$, as required.
	
	If $p\mid\frac{D_1D_2-x^2}{4}$, first assume that $p$ is odd. Then $p\mid x^2$, whence $p\mid x$, and $\frac{x}{p}$ is an integer with the same parity as $\frac{D_1D_2}{p^2}$, as required. 
	
	If $p=2$, then $8\mid D_1D_2-x^2$. If $D_2$ is even or $8\mid D_1$, then $8\mid D_1D_2$, so $8\mid x^2$, and hence $4\mid x$. Therefore $\frac{x}{2}\equiv 0\equiv \frac{D_1D_2}{2^2}\pmod{2}$, as required. Otherwise, $4\mid\mid D_1$ and $D_2$ is odd. As $D_1/4$ is a discriminant, it is equivalent to $1\pmod{4}$, and so $D_1D_2\equiv 4\pmod{16}$. This implies that $x^2\equiv 4\pmod{8}$, and so $x\equiv 2\pmod{4}$. Then $\frac{x}{2}\equiv 1\equiv\frac{D_1D_2}{2^2}\pmod{2}$, which completes the proof.
\end{proof}

We are now able to study the optimality of embeddings in $\Ord_{\phi_1,\phi_2}^{\Eich}$, as well as the level of this order.

\begin{definition}\label{def:potentiallybad}
	Let $D_1,D_2$ be discriminants. Define a prime $p$ to be \textit{potentially bad} (with respect to $D_1,D_2$) if
	\[p\mid\dfrac{D_1D_2}{D_1^{\fund}D_2^{\fund}}.\]
	Define $\PB(D_1,D_2)$ to be the product of all potentially bad primes. In particular, $D_1$ and $D_2$ are both fundamental if and only if $\PB(D_1,D_2)=1$.
\end{definition}

It suffices to consider the optimality of $(\phi_1, \phi_2)$ at potentially bad primes.

\begin{proposition}\label{prop:Eichlerlevelandoptimality}
	Let $\phi_1,\phi_2$ be $x-$linked embeddings of discriminants $D_1,D_2$ into $B$, an indefinite quaternion algebra over $\QQ$, where $(D_1,D_2,x)$ is \nice. Factorize
	\[\dfrac{D_1D_2-x^2}{4}=\pm\prod_{i=1}^{r}p_i^{2e_i+1}\prod_{i=1}^s q_i^{2f_i}\prod_{i=1}^t w_i^{g_i},\]
	where $p_i$ are the primes for which $\epsilon(p_i)=-1$ that appear to an odd power, $q_i$ are the primes for which $\epsilon(q_i)=-1$ that appear to an even power, and $w_i$ are the primes for which $\epsilon(w_i)=1$. Then
	\begin{enumerate}[label=(\roman*)]
		\item The order $\Ord_{\phi_1,\phi_2}^{\Eich}$ is Eichler of level $\prod_{i=1}^t w_i^{g_i}$;
		\item The embeddings $\phi_1,\phi_2$ are optimal embeddings into $\Ord_{\phi_1,\phi_2}^{\Eich}$ if and only if none of primes $p_i$ and $q_i$ are \pbad.
	\end{enumerate}
\end{proposition}
\begin{proof}
	Let $\Ord=\Ord_{\phi_1,\phi_2}$ and $\Ord^{\Eich}=\Ord_{\phi_1,\phi_2}^{\Eich}$. Lemma \ref{lem:genorder} computes the reduced discriminant of $\Ord$ to be $\frac{D_1D_2-x^2}{4}$, so it suffices to compute the change in reduced discriminant between $\Ord$ and $\Ord^{\Eich}$, which can be done locally. Lemma \ref{lem:Eichlernessoforder} implies that $\Ord_p=\Ord_p^{\Eich}$ for $p=w_i$, hence those prime factors remain in the level. For $p=p_i,q_i$, $\Ord_p$ is contained in a unique maximal order, hence those prime factors disappear. This completes the first point.
	
	For optimality, assume that $\phi_1$ is not optimal with respect to $\Ord^{\Eich}$. Thus there exists a $p\mid\frac{D_1}{D_1^{\fund}}$ for which $\phi(\qord{D_1/p^2})$ lands inside $\Ord^{\Eich}$. Let $\phi_1'$ denote this embedding (which agrees with $\phi$ on $\qordD$), and then $(\phi_1',\phi_2)$ are $\frac{x}{p}-$linked. By definition, we have 
	\[\Ord\subseteq \Ord_{\phi_1',\phi_2}\subseteq \Ord^{\Eich},\]
	and Lemma \ref{lem:genorder} says that the reduced discriminant of $\Ord_{\phi_1',\phi_2}$ is $\frac{D_1D_2-x^2}{4p^2}$. Therefore $p=p_i,q_i,w_i$, so assume that $p=w_i$. By Lemma \ref{lem:Eichlernessoforder}, $\Ord_p=\Ord_p^{\Eich}$, hence this is equal to $\Ord_{\phi_1',\phi_2,p}$ as well, which contradicts the fact that the level of $\Ord_{\phi_1',\phi_2}$ differs from the level of $\Ord$ by the factor $p^2$. Therefore $p=p_i$ or $p=q_i$, as claimed.
	
	To finish, it suffices to show that if $p\mid\frac{D_1}{D_1^{\fund}},\frac{D_1D_2-x^2}{4}$ satisfies $\epsilon(p)=-1$, then the embedding $\phi_1$ is not optimal into $\Ord^{\Eich}$. As above, let $\phi_1'$ denote the embedding of $\qord{D_1/p^2}$ corresponding to $\phi$. By Lemma \ref{lem:nonfundamentalxlink}, $(\phi_1',\phi_2)$ are $\frac{x}{p}-$linked, so by Lemma \ref{lem:genorder}, $\Ord_{\phi_1',\phi_2}$ is an order of reduced discriminant $\frac{D_1D_2-x^2}{4p^2}$. Since $\Ord\subseteq\Ord_{\phi_1',\phi_2}$ and $\Ord_p$ is contained in a unique maximal order, this must be the same maximal order that contains $\Ord_{\phi_1',\phi_2,p}$. Therefore $\Ord_p^{\Eich}=\Ord_{\phi_1',\phi_2,p}^{\Eich}$, and so $\phi_1'$ embeds into $\Ord_p^{\Eich}$, hence it embeds into $\Ord^{\Eich}$, which proves that $\phi_1$ is not optimal.
\end{proof}

To finish off with optimality, we need to consider the optimality of $\phi_1,\phi_2$ into superorders $\Ord'$ of $\Ord_{\phi_1,\phi_2}^{\Eich}$. Assume that none of the $p_i,q_i$ are potentially bad, so that $\phi_1,\phi_2$ are optimal in $\Ord_{\phi_1,\phi_2}^{\Eich}$. The only way that $\phi_1$ would fail optimality in $\Ord'$ is if $\Ord'$ admitted the embedding $\phi_1'$ of discriminant $\qord{D_1/w_j^2}$ (some $1\leq j\leq t$) that agrees with $\phi$ on $\qord{D_1}$. The pair $(\phi_1',\phi_2)$ is $\frac{x}{w_j}-$linked by Lemma \ref{lem:nonfundamentalxlink}, and $\Ord_{\phi_1',\phi_2}^{\Eich}$ is an Eichler order of level $\frac{1}{w_j^2}\prod_{i=1}^t w_i^{g_i}$ by Proposition \ref{prop:Eichlerlevelandoptimality}i. Therefore $\Ord'$ admits $\phi_1$ as an optimal embedding if and only if $\Ord'\not\supseteq\Ord_{\phi_1',\phi_2}^{\Eich}$.

\begin{definition}
	With notation and assumptions as above, let $S_{\phi_1,\phi_2}$ be the (possibly empty) set of orders $\Ord_{\phi_1',\phi_2}^{\Eich}$ and $\Ord_{\phi_1,\phi_2'}^{\Eich}$, each of which corresponds to a $1\leq j\leq t$ for which $w_j\mid\frac{D_1}{D_1^{\fund}}$ or $w_j\mid\frac{D_2}{D_2^{\fund}}$ respectively.
\end{definition}

The above discussion is the proof of the following proposition.

\begin{proposition}\label{prop:whensuperorderadmitsopt}
	Take the notation as in Proposition \ref{prop:Eichlerlevelandoptimality}, and assume that none of $p_i,q_i$ are \pbad. Then a superorder $\Ord'$ of $\Ord_{\phi_1,\phi_2}^{\Eich}$ admits $\phi_1,\phi_2$ as optimal embeddings if and only if it does not contain any order in $S_{\phi_1,\phi_2}$.
\end{proposition}

\subsection{Local x-linking with level}\label{sec:localxwithlevel}
Given an admissible triple $(D_1,D_2,x)$, Theorem \ref{thm:onequatalgxlink} determines the unique quaternion algebra for which there exists $x-$linked optimal embeddings. Under the additional restriction of niceness, Propositions \ref{prop:Eichlerlevelandoptimality} and \ref{prop:whensuperorderadmitsopt} determine the possible Eichler orders that an $x-$linked pair of embeddings becomes optimal in. In this section, we study the possible levels of such embeddings.

\begin{lemma}\label{lem:threeorders}
	Let $B$ be a quaternion algebra over $F=\QQ$ or $\QQ_p$. Let $v_1,v_2,v_3\in B$ be such that $\langle 1,v_i,v_j,v_iv_j\rangle_{\mathcal{O}_F}$ is an order for $(i,j)=(1,2),(1,3),(2,3)$. Then 
	\[\Ord=\langle 1,v_1,v_2,v_3,v_1v_2,v_1v_3,v_2v_3,v_1v_2v_3\rangle_{\mathcal{O}_F}\]
	is an order.
\end{lemma}
\begin{proof}
	It suffices to show that any product $v=v_{i_1}\cdots v_{i_k}$ lands in $\Ord$ for any sequence $i_1,\ldots,i_k$ with $i_j\in\{1,2,3\}$ for all $j$. This is accomplished via induction: the base case of $k=0$ is trivial. For the inductive step, assume it is true up to $k-1\geq 0$. If $i_j\neq 1$ for all $j$, then $v\in\langle 1,v_2,v_3,v_2v_3\rangle _{\mathcal{O}_F}$ (as this is an order), and we are done. Otherwise, take the last occurrence of $1$, say $i_m$. If $i_{m-1}=1$, then $v_{i_{m-1}}v_{i_m}=v_1^2\in \langle 1,v_1\rangle _{\mathcal{O}_F}$, and by replacing it we are done by induction. Otherwise, if $m>1$, then $i_{m-1}=j\neq 1$, hence $v_{i_{m-1}}v_{i_m}=v_jv_1\in \langle 1,v_1,v_j,v_1v_j\rangle _{\mathcal{O}_F}$. By writing $v_{i_{m-1}}v_{i_m}$ in this basis and using induction, we see that it suffices to prove the claim when we swap $v_{i_m}$ and $v_{i_{m-1}}$. By successively repeating this process, we can assume that $v$ starts with a $v_1$ and has no other terms $v_1$. But then $v_{i_2}\cdots v_{i_k}$ lies in $\langle 1,v_2,v_2,v_2v_3\rangle_{\mathcal{O}_F}$, and a left multiplication by $v_1$ still lands us in $\Ord$, as desired.
\end{proof}

The generalization of $\Ord_{\phi_1,\phi_2}$ is the following.

\begin{definition}
	Let $\phi_1,\phi_2$ be $x-$linked embeddings from $\qord{D_1},\qord{D_2}$ to $B$. Let $\ell\in\ZZ^+$ be such that $\frac{x^2-D_1D_2}{\ell^2}$ is a discriminant, and define $\Ord_{\phi_1,\phi_2}(\ell)$ to be the smallest order for which $\qord{D_1},\qord{D_2},\qord{(x^2-D_1D_2)/\ell^2}$ embed into via $\phi_1,\phi_2,\phi_1\times\phi_2$ respectively, if it exists. 
\end{definition}

\begin{lemma}\label{lem:genorderlevel}
	Let $F=\QQ$, and assume $(D_1, D_2, x)$ is \nice. Then $\Ord_{\phi_1,\phi_2}(\ell)$ exists if and only if $\ell^2\mid \frac{D_1D_2-x^2}{4}$, and when it does, it has reduced discriminant $\frac{D_1D_2-x^2}{4\ell^2}$.
\end{lemma}
\begin{proof}
	Let $D_3=\frac{x^2-D_1D_2}{\ell^2}$, and let $\phi_3:\qord{D_3}\rightarrow B$ be the embedding induced by $\phi_1\times\phi_2$. Let $w_i=\phi_i(\sqrt{D_i})$ and $v_i=\phi_i\left(\frac{p_{D_i}+\sqrt{D_i}}{2}\right)$ for $i=1,2,3$, and let $x=\frac{1}{2}\trd(w_1w_2)\equiv D_1D_2\pmod{2}$ by assumption. We have $w_3=\frac{w_1w_2-x}{\ell}$, whence
	\[\dfrac{1}{2}\trd(w_1w_3)=\dfrac{1}{2}\trd\left(\dfrac{D_1w_2-xw_1}{\ell}\right)=0.\]
	Similarly, $\frac{1}{2}\trd(w_2w_3)=0$. If $D_3$ is odd, then since $D_1$ or $D_2$ is odd, $p_{D_iD_3}=1\not\equiv 0\pmod{2}$ for $i=1$ or $2$, whence $\left\langle 1,v_i,v_3,v_iv_3\right\rangle_{\ZZ}$ is not an order, and $\Ord_{\phi_1,\phi_2}(\ell)$ does not exist. Since $D_3$ is a discriminant, if it is not odd it must be a multiple of $4$. In particular, we have that $\ell^2\mid \frac{D_1D_2-x^2}{4}$. In this case, $0\equiv D_iD_3\pmod{2}$ for $i=1,2$, and so by Lemma \ref{lem:genorder}, $\langle 1,v_i,v_j,v_iv_j\rangle_{\ZZ}$ is an order for $(i,j)=(1,2),(1,3),(2,3)$. Thus by Lemma \ref{lem:threeorders}, $\Ord=\langle 1,v_1,v_2,v_3,v_1v_2,v_1v_3,v_2v_3,v_1v_2v_3\rangle_{\ZZ}$ is an order, necessarily the smallest order for which $\phi_i$ embeds into for all $i=1,2,3$.
	
	Let $p_i=p_{D_i}$, and compute
	\[\left(\begin{matrix}1\\v_1\\v_2\\v_3\\v_1v_2\\v_1v_3\\v_2v_3\\v_1v_2v_3\end{matrix}\right)=\left(\begin{matrix} 1 & 0 & 0 & 0\\ \frac{p_1}{2} & \frac{1}{2} & 0 & 0\\ \frac{p_2}{2} & 0 & \frac{1}{2} & 0\\0 & 0 & 0 & \frac{1}{2} \\\frac{p_1p_2+x}{4} & \frac{p_2}{4} & \frac{p_1}{4} & \frac{\ell}{4}\\0 & \frac{-x}{4\ell} & \frac{D_1}{4\ell} & \frac{p_1}{4}\\0 & \frac{-D_2}{4\ell} & \frac{x}{4\ell} & \frac{p_2}{4}\\ \frac{x^2-D_1D_2}{8\ell} & \frac{-p_2x-p_1D_2}{8\ell} & \frac{p_1x+p_2D_1}{8\ell} & \frac{p_1p_2+x}{8}\\ \end{matrix}\right)\left(\begin{matrix}1\\w_1\\w_2\\w_3\end{matrix}\right).\]
	Let this transition matrix be $M$. From the calculation in Lemma \ref{lem:genorder}, we can compute that $d(1,w_1,w_2,w_3)=\left(\frac{1}{\ell}\right)^2d(1,w_1,w_2,w_1w_2)=-\frac{16(D_1D_2-x^2)^2}{\ell^2}$. It suffices to show that the rows of $M$ generate a $\ZZ-$lattice with determinant $\frac{1}{16\ell}$, as then we have the discriminant of $\Ord$ being $\frac{(D_1D_2-x^2)^2}{16\ell^4}$, whence the reduced discriminant is $\frac{D_1D_2-x^2}{4\ell^2}$, as desired. The calculation of the rowspace is done by hand in Appendix \ref{app:hnfcalc}.
\end{proof}

\begin{remark}
	The statement $\ell^2\mid\frac{D_1D_2-x^2}{4}$ only requires $(D_1, D_2, x)$ to be \nice\, at $p=2$. If it is not nice at $p=2$, then this does not need to hold. For example, take $D_1=20$, $D_2=68$, $x=2$, and $B$ to be ramified at $3, 113$. Then $\Ord_{\phi_1, \phi_2}(2)$ exists, but $2^2\nmid\frac{D_1D_2-x^2}{4}=339$.
\end{remark}

Since $\Ord_{\phi_1,\phi_2}\subseteq\Ord_{\phi_1,\phi_2}(\ell)$, the inclusion holds when we complete at $p$. Considering Lemma \ref{lem:Eichlernessoforder}, we find that
\begin{itemize}
	\item If $p\nmid\frac{D_1D_2-x^2}{4}$, then $\Ord_{\phi_1,\phi_2,p}(\ell)$ is maximal;
	\item If $\epsilon(p)=-1$, then $\Ord_{\phi_1,\phi_2,p}(\ell)$ is contained in a unique maximal order, necessarily the same maximal order as the one containing $\Ord_{\phi_1,\phi_2,p}$;
	\item If $\epsilon(p)=1$, then $\Ord_{\phi_1,\phi_2,p}(\ell)$ is Eichler.
\end{itemize}
In particular, this implies that there exists a minimal Eichler order containing $\Ord_{\phi_1,\phi_2}(\ell)$, denoted $\Ord_{\phi_1,\phi_2}^{\Eich}(\ell)$. Factorize
\[\dfrac{D_1D_2-x^2}{4}=\pm\prod_{i=1}^{r}p_i^{2e_i+1}\prod_{i=1}^s q_i^{2f_i}\prod_{i=1}^t w_i^{g_i},\]
where $p_i$ are the primes for which $\epsilon(p_i)=-1$ that appear to an odd power, $q_i$ are the primes for which $\epsilon(q_i)=-1$ that appear to an even power, and $w_i$ are the primes for which $\epsilon(w_i)=1$. The local conditions imply that
\[\Ord_{\phi_1,\phi_2}^{\Eich}=\Ord_{\phi_1,\phi_2}^{\Eich}\left(\prod_{i=1}^r p_i^{e_i}\prod_{i=1}^s q_i^{f_i}\right),\]
i.e. that the maximum possible level always occurs at the prime factors $p$ of $\frac{D_1D_2-x^2}{4}$ for which $\epsilon(p)=-1$. The analogous assessment of the prime factors $p$ for which $\epsilon(p)=1$ leads to the following proposition.

\begin{proposition}\label{prop:optembwithlevel}
	Let $(D_1, D_2, x)$ be \nice, and let $\ell=\prod_{i=1}^{r}p_i^{e_i'}\prod_{i=1}^s q_i^{f_i'}\prod_{i=1}^t w_i^{g_i'}$, where $e_i'\leq e_i$, $f_i'\leq f_i$, and $2g_i'\leq g_i$. Then the Eichler order $\Ord_{\phi_1,\phi_2}^{\Eich}(\ell)$ has level $\prod_{i=1}^t w_i^{g_i-2g_i'}$. Furthermore, assume that all the $p_i,q_i$ are not \pbad. Let 
	\[S=\left\{w_i:w_i\mid\PB(D_1,D_2)\right\}\]
	be the set of potentially bad primes among the $w_i$. Then a superorder $\Ord'$ of $\Ord_{\phi_1,\phi_2}^{\Eich}(\ell)$ admits $\phi_1,\phi_2$ as optimal embeddings if and only if $\Ord'$ does not contain $\Ord_{\phi_1,\phi_2}^{\Eich}(w_i)$ for all $w_i\in S$. This implies $g_i'=0$ for all $i$ such that $w_i\in S$.
\end{proposition}
\begin{proof}
	The first half of the proposition has been proven in the above discussion. For the second half, the optimality of $\phi_1,\phi_2$ can only fail if we have a $w_j$ for which $\phi_1$ (without loss of generality) descends to an embedding of $\qord{D_1/w_j^2}$. Call this embedding $\phi_1'$, and as in Proposition \ref{prop:Eichlerlevelandoptimality} the order $\Ord_{\phi_1',\phi_2}^{\Eich}$ has level $\frac{1}{w_j^2}\prod_{i=1}^t w_i^{g_i}$. It suffices to show that $\Ord_{\phi_1',\phi_2,w_j}=\Ord_{\phi_1,\phi_2,w_j}(w_j)$, as this means that picking up a factor of $w_j$ in the level is equivalent to killing optimality.
	
	These Eichler orders have the same level, so it suffices to show inclusion only. However this is immediate, as the embedding $\phi_1'\times\phi_2$ corresponds to an embedding of discriminant $\frac{x^2-D_1D_2}{p^2}$ induced from 
	\[\phi_1'(\sqrt{D_1/p^2})\phi_2(\sqrt{D_2})=\frac{1}{p}\phi_1(\sqrt{D_1})\phi_2(\sqrt{D_2}).\]
\end{proof}

An embedding pair having level exactly $\ell$ in $\Ord'$ is equivalent to $\Ord'$ containing $\Ord_{\phi_1,\phi_2}(\ell)$ but not containing $\Ord_{\phi_1,\phi_2}(p\ell)$ for any prime $p$. At long last, we can describe the levels and counts of Eichler orders admitting $\phi_1,\phi_2$ as optimal embeddings.

\begin{theorem}\label{thm:generalxlinkingexistence}
	Let $\phi_1,\phi_2$ be $x-$linked embeddings of discriminants $D_1,D_2$ into $B$, an indefinite quaternion algebra over $\QQ$, let $\ell$ be a positive integer, and assume that $(D_1,D_2,x)$ is \nice. Factorize
	\[\dfrac{D_1D_2-x^2}{4}=\pm\prod_{i=1}^{r}p_i^{2e_i+1}\prod_{i=1}^s q_i^{2f_i}\prod_{i=1}^t w_i^{g_i},\]
	where $p_i$ are the primes for which $\epsilon(p_i)=-1$ that appear to an odd power, $q_i$ are the primes for which $\epsilon(q_i)=-1$ that appear to an even power, and $w_i$ are the primes for which $\epsilon(w_i)=1$. Then,
	\begin{enumerate}[label=(\roman*)]
		\item This setup is possible if and only if $B$ is ramified at exactly $p_1,p_2,\ldots p_r$;
		\item There exists an Eichler order of level $\mathfrak{M}$ for which $\phi_1,\phi_2$ are optimal embeddings into if and only if both of the following are satisfied:
		\begin{itemize}
			\item None of the $p_i,q_i$ are \pbad;
			\item $\mathfrak{M}=\prod_{i=1}^t w_i^{g_i'}$ with $g_i'\leq g_i$.
		\end{itemize}
		\item Let $\mathfrak{M}$ satisfy the above. The number of Eichler orders of level $\mathfrak{M}$ for which $\phi_1,\phi_2$ are optimal embeddings into is
		\[\displaystyle\prod_{i=1}^t
		\begin{cases} g_i+1-g_i' & \text{if $w_i\nmid\PB(D_1,D_2)$;}\\
			2 & \text{if $w_i\mid\PB(D_1,D_2)$ and $g_i'<g_i$;}\\
			1 & \text{if $w_i\mid\PB(D_1,D_2)$ and $g_i'=g_i$.}
		\end{cases}\]
		\item There exists an Eichler order of level $\mathfrak{M}$ for which $\phi_1,\phi_2$ are optimal embeddings of into of level exactly $\ell$ if and only we have
		\[\ell=\prod_{i=1}^{r}p_i^{e_i}\prod_{i=1}^s q_i^{f_i}\prod_{i=1}^t w_i^{g_i''},\]
		where $2g_i''\leq g_i-g_i'$ and $g_i''=0$ if $w_i\mid\PB(D_1,D_2)$.
		\item Let $\mathfrak{M},\ell$ satisfy the above. Let $n$ be the number of indices $i$ for which $2g_i''<g_i-g_i'$. Then the number of Eichler orders of level $\mathfrak{M}$ for which $\phi_1,\phi_2$ are optimal embeddings into of level exactly $\ell$ is $2^n$.
	\end{enumerate}
\end{theorem}
\begin{proof}
	Part i is the content of Theorem \ref{thm:onequatalgxlink}, and the necessity of the conditions in part ii follows from Proposition \ref{prop:Eichlerlevelandoptimality}. To complete part ii, it suffices to prove it locally, and Proposition \ref{prop:whensuperorderadmitsopt} implies that there is an Eichler order of level $w_i^{g_i-2}$ whose containment must be avoided for each $i$ such that $w_i\mid\PB(D_1,D_2)$ (and no other orders need be avoided).
	
	Recall the inverted triangle of local Eichler orders, as described in Section \ref{sec:towersEichler}. The local Eichler orders containing $\Ord_{\phi_1,\phi_2,w_i}^{\Eich}$ form an inverted triangle with $g_i+1$ rows. There are $g_i+1-n$ Eichler orders of level $w_i^n$ in the $n^{\text{th}}$ row of the triangle, starting at $n=0$ and ending at $n=g_i$. Therefore if $w_i\nmid\PB(D_1,D_2)$, there are $g_i+1-g_i'$ possible Eichler orders of level $w_i^{g_i'}$. If $w_i\mid\PB(D_1,D_2)$, then there is one when $g_i'=g_i$, and on all rows above it there are two, as the order that we cannot contain has level $w_i^{g_i-2}$. In particular, this implies part ii as this is a non-zero number.
	
	By the local-global principle for orders (Theorem \ref{thm:JVlocalglobal}), the total count for global orders is the product of the local counts. The count in part iii follows from this and the previous paragraph.
	
	For parts iv, v, Proposition \ref{prop:optembwithlevel} and the discussion surrounding it imply that $\ell$ has the prime factorization as claimed. The necessity of $2g_i''\leq g_i-g_i'$ comes from the level of $\Ord_{\phi_1,\phi_2}^{\Eich}(w_i^{g_i''})$ having valuation $g_i-2g_i''$ at $w_i$. Proposition \ref{prop:optembwithlevel} also implies that if $w_i\mid\PB(D_1,D_2)$, then the valuation of $\ell$ at $w_i$ must be $0$, i.e. $g_i''=0$. 
	
	To count this, we again work locally and use the local-global principle. The local count is unchanged at the primes $w_i$ for which $w_i\mid\PB(D_1,D_2)$. For primes $w_i$ not satisfying this, we no longer have to worry about optimality. For ease of notation, if the level of the embedding pair is $w_i^k$, we say it has intersection level $k$. The Eichler order $\Ord_{\phi_1,\phi_2,w_i}^{\Eich}(w_i^n)$ has level $w_i^{g_i-2n}$, and an intersection level is at least $n$ if and only if the order contains $\Ord_{\phi_1,\phi_2,w_i}^{\Eich}(w_i^n)$. Drawing the inverted triangle as before, it follows by induction that (noting that all of the orders $\Ord_{\phi_1,\phi_2,w_i}^{\Eich}(w_i^n)$ are contained inside each other)
	\begin{itemize}
		\item In level $w_i^{g_i-2n}$, there are $2n+1$ orders, of which there are $2$ of each intersection level $0,1,\ldots,n-1$, and one of intersection level $n$;
		\item In level $w_i^{g_i-2n+1}$, there are $2n$ orders, of which there are $2$ of each intersection level $0,1,\ldots,n-1$.
	\end{itemize}
	In particular, there are $2$ orders of intersection level $g_i''$ when $2g_i''<g_i-g_i'$, and one when $2g_i''=g_i-g_i'$. The condition coming from $w_i\mid\PB(D_1,D_2)$ was there are two if $g_i'<g_i$, and one if we had equality. Since $g_i''=0$, this condition is absorbed by $2g_i''<g_i-g_i'$. This completes parts iv, v.
\end{proof}

\section{Proof of the main theorem}\label{sec:mainproof}

We are now ready to study $\Emb(\Ord,D_1,D_2,x)$.

\subsection{Total x-linking into a given Eichler order}

As alluded to at the start of Section \ref{sec:xlink}, we need to pass between Eichler orders containing a fixed pair of $x-$linked embeddings, and elements of $\Emb(\Ord,D_1,D_2,x)$. This is accomplished in the ``inversion theorem'', which we now set up for.

Let $F$ be $\QQ $ or $\QQ_p$, and let $B$ be a quaternion algebra over $F$ of discriminant $\mathfrak{D}$, which is indefinite if $F=\QQ $. Let $\Ord$ be an Eichler order of level $\mathfrak{M}$ in $B$. Assume that $D_1,D_2$ are positive discriminants for which $\Emb(B,D_1,D_2,x)$ is non-empty, fix $[(\phi_1,\phi_2)]\in\Emb(B,D_1,D_2,x)$, let $\ell^2\mid\frac{D_1D_2-x^2}{4}$, and define
\begin{align*}
	T_{\phi_1,\phi_2}(\mathfrak{M}):= & \{\mathrm{E}:\text{ $\mathrm{E}$ is an Eichler order of $B$ of level $\mathfrak{M}$ }\\
	& \qquad\qquad \text{for which $\phi_1,\phi_2$ give optimal embeddings into}\};\\
	T_{\phi_1,\phi_2}(\mathfrak{M},\ell):= & \{\mathrm{E}\in T_{\phi_1,\phi_2}(\mathfrak{M})\text{ such that $(\phi_1,\phi_2)$ has level $\ell$ in $\mathrm{E}$}\}.
\end{align*}

\begin{proposition}\label{prop:countemb}
	We have
	\[\vert\Emb(\Ord,D_1,D_2,x,\ell)\vert=\left\vert\dfrac{N_{B^{\times}}(\Ord)}{F^{\times}\Ord^1}\right\vert\vert T_{\phi_1,\phi_2}(\mathfrak{M},\ell)\vert,\]
	and the analogous result without the $\ell$. If $F=\QQ$, then
	\[\left\vert\dfrac{N_{B^{\times}}(\Ord)}{F^{\times}\Ord^1}\right\vert=2^{\omega(\mathfrak{D}\mathfrak{M})+1}.\]
\end{proposition}
\begin{proof}
	By Corollary \ref{cor:allpairsconj} and Eichler orders of the same level being conjugate, we have that $T_{\phi_1,\phi_2}(\mathfrak{M})$ is non-empty if and only if $S=\Emb(\Ord,D_1,D_2,x)$ is non-empty. In particular, we can assume that $(\phi_1,\phi_2)$ give a class in $S$, and we will use this pair to define a map $\theta:S\rightarrow T_{\phi_1,\phi_2}(\mathfrak{M})$. Given optimal embeddings $(\phi_1',\phi_2')$ representing a class in $S$, by Corollary \ref{cor:allpairsconj}, there exists an $r\in B^{\times}$ for which $r\phi_i'r^{-1}=\phi_i$ for $i=1,2$. Define
	\[\theta((\phi_1',\phi_2'))=r\Ord r^{-1}.\]
	It is clear that $r\Ord r^{-1}\in T_{\phi_1,\phi_2}(\mathfrak{M})$, but we need to check that all choices were well defined. By Corollary \ref{cor:allpairsconj}, the element $r$ is defined up to multiplication by $F^{\times}$, which does not change $r\Ord r^{-1}$. If $(\phi_1',\phi_2')\sim(\phi_1'',\phi_2'')$ in $S$, then there exists an $s\in\Ord^1$ for which $\phi_i'=s\phi_i''s^{-1}$ for $i=1,2$. The corresponding element $r$ can then be taken to be $r'=rs$, and then $r'\Ord r'^{-1}=rs\Ord s^{-1}r^{-1}=r\Ord r^{-1}$, as desired. Therefore the map $\theta$ is well defined.
	
	Next, it is clear that $\theta$ is surjective. Indeed, if $\mathrm{E}\in T_{\phi_1,\phi_2}(\mathfrak{M})$, then there exists a $b\in B^{\times}$ for which $b\mathrm{E}b^{-1}=\Ord$. Then $(\phi_1^b,\phi_2^b)\in S$, and this pair maps via $\theta$ to $\mathrm{E}$, as desired.
	
	Therefore, it suffices to show that $\theta$ is a $\left\vert\frac{N_{B^{\times}}(\Ord)}{F^{\times}\Ord^1}\right\vert$-to-one map. Assume that $\theta((\phi_1',\phi_2'))=\theta((\phi_1'',\phi_2''))$, and that the pairs correspond to $r,s$ respectively. Then $r\Ord r^{-1}=s\Ord s^{-1}$, hence $t=r^{-1}s\in N_{B^{\times}}(\Ord)$. Writing $s=rt$, it follows that $t^{-1}\phi_i't=\phi_i''$, so it suffices to determine how $t^{-1}(\phi_1',\phi_2')t$ varies as $t$ ranges over $N_{B^{\times}}(\Ord)$. For a fixed $t$, by Corollary \ref{cor:allpairsconj}, the set of elements conjugating $(\phi_1',\phi_2')$ to any form in the class of $t^{-1}(\phi_1',\phi_2')t$ is $\Ord^1 t^{-1}F^{\times}=t^{-1}F^{\times}\Ord^1$. Thus, for distinct $t_1,t_2$, they correspond to the same image if and only if 
	\[t_1^{-1}F^{\times}\Ord^1=t_2^{-1}F^{\times}\Ord^1,\]
	which is equivalent to $t_2t_1^{-1}\in F^{\times}\Ord^1$. This proves the first claim without the $\ell$. It is clear that the level of intersection remains constant under $\theta$, hence the statements remain true when we add in the level $\ell$.
	
	When $F=\QQ$, Proposition \ref{prop:stab} yields
	\[\dfrac{N_{B^{\times}}(\Ord)}{\QQ^{\times}\Ord^1}\simeq\prod_{p\mid\mathfrak{D}\mathfrak{M}\infty}\dfrac{\ZZ}{2\ZZ},\]
	which implies the final result.
\end{proof}

Combining Proposition \ref{prop:countemb} with Theorem \ref{thm:generalxlinkingexistence} produces the count of $x-$linking.

\begin{theorem}\label{thm:maincountingthm}
	Let $B$ be an indefinite quaternion algebra over $\QQ$ of discriminant $\mathfrak{D}$, let $\Ord$ be an Eichler order of level $\mathfrak{M}$, let $(D_1,D_2,x)$ be \nice, and let $\ell$ be a positive integer. Factorize
	\[\dfrac{D_1D_2-x^2}{4}=\pm\prod_{i=1}^{r}p_i^{2e_i+1}\prod_{i=1}^s q_i^{2f_i}\prod_{i=1}^t w_i^{g_i},\]
	where the $p_i$ are the primes for which $\epsilon(p_i)=-1$ that appear to an odd power, $q_i$ are the primes for which $\epsilon(q_i)=-1$ that appear to an even power, and $w_i$ are the primes for which $\epsilon(w_i)=1$. Then
	\begin{enumerate}[label=(\roman*)]
		\item The set $\Emb(\Ord,D_1,D_2,x)$ is non-empty if and only if all of the following hold:
		\begin{itemize}
			\item $\mathfrak{D}=\prod_{i=1}^r p_i$;
			\item None of the $p_i,q_i$ are \pbad;
			\item $\mathfrak{M}=\prod_{i=1}^t w_i^{g_i'}$ with $g_i'\leq g_i$.
		\end{itemize}
		\item Assume the above holds. Then
		\[\vert\Emb(\Ord,D_1,D_2,x)\vert=2^{\omega(\mathfrak{D}\mathfrak{M})+1}\displaystyle\prod_{i=1}^t
		\begin{cases} g_i+1-g_i' & \text{if $w_i\nmid\PB(D_1,D_2)$;}\\
			2 & \text{if $w_i\mid\PB(D_1,D_2)$ and $g_i'<g_i$;}\\
			1 & \text{if $w_i\mid\PB(D_1,D_2)$ and $g_i'=g_i$.}
		\end{cases}\]
		\item The set $\Emb(\Ord,D_1,D_2,x,\ell)$ is non-empty if and only if $\ell$ takes the form
		\[\ell=\prod_{i=1}^{r}p_i^{e_i}\prod_{i=1}^s q_i^{f_i}\prod_{i=1}^t w_i^{g_i''},\]
		where $2g_i''\leq g_i-g_i'$ and $g_i''=0$ if $w_i\mid\PB(D_1,D_2)$.
		\item Assume the above holds. Let $n$ be the number of indices $i$ for which $2g_i''<g_i-g_i'$. Then
		\[\vert\Emb(\Ord,D_1,D_2,x,\ell)\vert=2^{\omega(\mathfrak{D}\mathfrak{M})+n+1}.\]
	\end{enumerate}
\end{theorem}

Most of Theorem \ref{thm:countxlinking} now follows by specializing Theorem \ref{thm:maincountingthm} to the case of $D_1,D_2$ being coprime, fundamental, and $\Ord$ being maximal. The only unproven claim is the final one about the signs of intersections, and this is considered in the next section.

\subsection{Orientations and sign of intersection}\label{sec:orientint}
Up until now, the orientations of optimal embeddings and the sign of intersection has been completely ignored; we now address this issue.

\begin{lemma}\label{lem:xlinkedor}
	Let $B$ be an indefinite quaternion algebra over $\QQ$, let $\Ord$ be an Eichler order of level $\mathfrak{M}$, let $(\phi_1,\phi_2)$ be $x-$linked optimal embeddings of positive discriminants $D_1,D_2$ respectively where $(D_1, D_2, x)$ is admissible, let $v\mid\mathfrak{D}\mathfrak{M}\infty$, and let $\omega_v\in N_B^{\times}(\Ord)$ be as in Proposition \ref{prop:stab}. Then $(\phi_1^{w_v},\phi_2^{w_v})$ is an $x-$linked pair of optimal embeddings into $\Ord$ with the same level as $(\phi_1,\phi_2)$. Furthermore, if $x^2<D_1D_2$, then
	\begin{itemize}
		\item If $v=\infty$ then the orientations are the same, but the sign of intersection is opposite.
		\item If $v<\infty$, then the orientations are negated at $v$ only, and the sign of intersection is the same.
	\end{itemize}
\end{lemma}
\begin{proof}
	It is clear that $(\phi_1^{w_v},\phi_2^{w_v})$ remains $x-$linked, optimal, has the same intersection level, and the orientation follows from Proposition \ref{prop:conjorient}. Having opposite sign of intersection is equivalent to $\phi_1\times\phi_2$ swapping orientation at $\infty$ when conjugating by $w_v$, and this also follows from Proposition \ref{prop:conjorient}.
\end{proof}

In particular, any element of $\Ord^{-1}$ (reduced norm $-1$) acts as an involution on $\Emb_{o_1,o_2}(\Ord,D_1,D_2,x,\ell)$, dividing it into equal sized sets of intersection sign being $1$ and $-1$. This completes the final claim of Theorem \ref{thm:countxlinking}.

\begin{definition}
	If $o_1,o_2$ are orientations of optimal embeddings, then attaching the subscript $o_1,o_2$ to any of the sets defined as $\Emb(\Ord, D_1, D_2,\ldots)$ means we only take the pairs of optimal embeddings of the specified orientations. Thus, $\Emb_{o_1,o_2}^+(\Ord,D_1,D_2,x,\ell)$ counts the equivalence classes of pairs $[(\phi_1,\phi_2)]$ of optimal embeddings of discriminants $D_1,D_2$ and orientations $o_1,o_2$ that are $x-$linked of level $\ell$ with positive sign.
\end{definition}

\begin{lemma}\label{lem:xlinkingdivisionamongorientations}
	Let $B$ be an indefinite quaternion algebra over $\QQ$ of discriminant $\mathfrak{D}$, let $\Ord$ be an Eichler order of level $\mathfrak{M}$, and let $(D_1, D_2, x)$ be admissible. Assume that $\Emb(\Ord,D_1,D_2,x)$ is non-empty, let $o_1$ be a possible orientation of an optimal embedding of $\qord{D_1}$ into $\Ord$, and assume that $\gcd(D_1D_2, \mathfrak{M})=1$. Then there exists a $[(\phi_1,\phi_2)]\in \Emb(\Ord,D_1,D_2,x)$ for which $\phi_1$ has orientation $o_1$. For each $p\mid\mathfrak{D}\mathfrak{M}$, we also have:
	\begin{itemize}
		\item If $p\nmid D_1$, then $o_p(\phi_2)$ is uniquely determined;
		\item If $p\mid D_1$ but $p\nmid D_2$, then $o_p(\phi_2)$ can be both $1$ and $-1$.
	\end{itemize}
	Finally, there is a positive integer $N$ such that for all orientations $(o_1,o_2)$, we have
	\[\vert\Emb_{o_1,o_2}(\Ord,D_1,D_2,x,\ell)\vert\in\{0,N\},\]
	and the same result holds with $N/2$ for $\Emb_{o_1,o_2}^+(\Ord,D_1,D_2,x,\ell)$ if $x^2<D_1D_2$.
\end{lemma}
\begin{proof}
	Start with $[(\phi_1',\phi_2')]\in \Emb(\Ord,D_1,D_2,x)$, and from Lemma \ref{lem:xlinkedor} we can conjugate the pair by $\omega_p$ for $p\mid\mathfrak{D}\mathfrak{M}$ to get $(\phi_1,\phi_2)$ with $\phi_1$ having orientation $o_1$.
	
	If $p\mid D_1$ but $p\nmid D_2$, the local orientation result follows from from conjugating the embeddings by $\omega_p$, as $o_p(\phi_1)=0$.
	
	Next, assume $p\nmid D_1$. It suffices to prove this lemma locally, so first assume we have $p\mid\mathfrak{D}$, i.e. $B_p$ is division. As in the proof of Lemma \ref{lem:localdivisionxlinking}, write $B_p=\left(\frac{D_1,e}{\QQ_p}\right)$, with $\phi_{1,p}(\sqrt{D_1})=i$, $(D_1,e)_p=-1$, and $v_p(e)$ being necessarily odd. Assume $p$ is odd, let $\phi_{2,p}(\sqrt{D_2})=fi+gj+hk$ for $f,g,h\in\ZZ_p$, and the trace condition gives that $f=\frac{x}{D_1}$. Let $\mathfrak{p}$ be the maximal order in $\Ord_p$, and since $p\mid\nrd(j),\nrd(k)$,
	\[\phi_{2,p}(\sqrt{D_2})\equiv \frac{x}{D_1}i\pmod{\mathfrak{p}},\]
	which only depends on $x,D_1$. Therefore by Lemma \ref{lem:sameorlocalconditions}, the local orientation of $\phi_2$ at $p$ is fixed. If $p=2$, then the analogous computations involving $\phi_{2,p}\left(\frac{p_{D_2}+\sqrt{D_2}}{2}\right)$ imply the result.
	
	Otherwise, assume that $B_p=\Mat(2,\QQ_p)$, and $\Ord_p$ is the standard Eichler level of order $p^e$ with $e>0$. Let $e_1=e+v_2(p)$, and then working modulo $p^{e_1}$ we write
	\[\phi_1(\sqrt{D_1})\equiv\left(\begin{matrix}a&b\\0&-a\end{matrix}\right)\pmod{p^{e_1}},\qquad\phi_2(\sqrt{D_2})\equiv\left(\begin{matrix}c&d\\0&-c\end{matrix}\right)\pmod{p^{e_1}}.\]
	Therefore $x\equiv ac\pmod{p^{e_1}}$, and since $p\nmid a$ (else $p\mid D_1$), we have $c\equiv\frac{x}{a}\pmod{p^{e_1}}$. By Lemma \ref{lem:sameorlocalconditions}, the local orientation of $\phi_2$ at $p$ is fixed.
	
	Finally, the above shows that we can pass between all pairs $(o_1,o_2)$ for which $\Emb_{o_1,o_2}(\Ord, D_1, D_2, x, \ell)$ is non-empty via conjugation by $\omega_p$ for $p\mid\mathfrak{D}\mathfrak{M}$, hence these sets all have the same size. If $x^2<D_1D_2$, then exactly half of the pairs in a given set have positive intersection sign, which completes the lemma.
\end{proof}

We can say even more about how the possible $x$'s divide across a pair of orientations.

\begin{proposition}\label{prop:orxsinglemod2dm}
	Let $B$ be an indefinite quaternion algebra over $\QQ$ of discriminant $\mathfrak{D}$, let $\Ord$ be an Eichler order of level $\mathfrak{M}$, let $D_1,D_2$ be positive discriminants, and let $o_1,o_2$ be possible orientations of optimal embeddings of discriminants $D_1,D_2$ into $\Ord$. Then there exists an integer $x_{o_1,o_2}$ such that for all optimal embeddings $\phi_i\in\Emb_{o_i}(\Ord,D_i)$ ($i=1,2$), we have
	\[x_{o_1,o_2}\equiv\frac{1}{2}\trd\left(\phi_1(\sqrt{D_1})\phi_2(\sqrt{D_2})\right)\pmod{2\mathfrak{D}\mathfrak{M}}.\]
	In particular, the possible $x-$linkings across an orientation pair are all equivalent modulo $2\mathfrak{D}\mathfrak{M}$.
\end{proposition}
\begin{proof}
	Fix another pair $\phi_i'\in\Emb_{o_i}(\Ord,D_i)$, and say that $\phi_1,\phi_2$ are $x-$linked and $\phi_1',\phi_2'$ are $x'-$linked. It suffices to show that $x\equiv x'\pmod{2\mathfrak{D}\mathfrak{M}}$. We can work locally, so start with $p\mid\mathfrak{D}$, and assume that $\phi_i,\phi_i'$ now land in $\Ord_p$. Let $\mathfrak{p}$ be the unique maximal order of $\Ord_p$, and as the embeddings have the same orientation, there exists $u_1,u_2\in\Ord_p^1$ for which $\phi_i'=\phi_i^{u_i}$ for $i=1,2$. Since $\Ord_p/\mathfrak{p}\simeq\mathbb{F}_{p^2}$ is commutative, when working modulo $\mathfrak{p}$ we can rearrange terms freely. Thus
	\[
		\phi_1^{u_1}\left(\frac{p_{D_1}+\sqrt{D_1}}{2}\right)\phi_2^{u_2}\left(\frac{p_{D_2}+\sqrt{D_2}}{2}\right)\equiv 
		\phi_1\left(\frac{p_{D_1}+\sqrt{D_1}}{2}\right)\phi_2\left(\frac{p_{D_2}+\sqrt{D_2}}{2}\right)\pmod{\mathfrak{p}}.
	\]
	Taking reduced traces implies that
	\[\dfrac{p_{D_1}p_{D_2}+x'}{2}\equiv\dfrac{p_{D_1}p_{D_2}+x}{2}\pmod{\mathfrak{p}}.\]
	If $p\neq 2$, it follows that $x'\equiv x\pmod{\mathfrak{p}}$, whence $x'\equiv x\pmod{p}$ by subtracting and taking the norm. If $p=2$, then $x'\equiv x\pmod{2\mathfrak{p}}$, and so subtracting and taking norms gives $8\mid (x'-x)^2$, hence $x\equiv x'\pmod{4}$.
	
	Next, assume that $p^e\mid\mid\mathfrak{M}$ with $e>0$, and assume that $\Ord_p$ is the standard Eichler order of level $p^e$. As the embeddings have the same orientation, there exists $u_1,u_2\in\Ord_p^1$ for which $\phi_i'=\phi_i^{u_i}$ for $i=1,2$. Explicitly write
	\begin{equation}\label{eqn:conjugateembor}
		\phi_i\left(\frac{p_{D_i}+\sqrt{D_i}}{2}\right)=\left(\begin{matrix} a_i & b_i\\p^ec_i & p_{D_i}-a_i\end{matrix}\right),\qquad u_i=\left(\begin{matrix} f_i & g_i\\p^eh_i & k_i\end{matrix}\right).
	\end{equation}
	It follows that $f_ik_i\equiv 1\pmod{p^e}$. Modulo $p^e$, we compute
	\begin{equation}\label{eqn:conjugateembor2}
		\phi_i^{u_i}\left(\frac{p_{D_i}+\sqrt{D_i}}{2}\right)\equiv\left(\begin{matrix} a_i & f_i(p_{D_i}g_i-2g_ia_i+f_ib_i)\\0 & p_{D_i}-a_i\end{matrix}\right)\pmod{p^e}.
	\end{equation}
	By taking the explicit expressions for $\phi_i\left(\frac{p_{D_i}+\sqrt{D_i}}{2}\right)$, doubling and subtracting $p_{D_i}$, and multiplying together, we find that
	\[x\equiv(2a_1-p_{D_1})(2a_2-p_{D_2})\equiv x'\pmod{p^{e+v_2(p)}},\]
	as claimed.
	
	Combining the above shows that $x\equiv x'\pmod{2\mathfrak{D}\mathfrak{M}}$ if $2\mid\mathfrak{D}\mathfrak{M}$, and $x\equiv x'\pmod{\mathfrak{D}\mathfrak{M}}$ otherwise. In this case, $x\equiv p_{D_1}p_{D_2}\equiv x'\pmod{2}$, so the same conclusion follows.
\end{proof}

If $D_1$ is coprime to $\mathfrak{D}\mathfrak{M}$, then Lemma \ref{lem:xlinkingdivisionamongorientations} and Proposition \ref{prop:orxsinglemod2dm} can be used to show that for $o_1$ fixed, the integers $x_{o_1,o_2}$ are all distinct modulo $2\mathfrak{D}\mathfrak{M}$ across all orientations $o_2$. If $D_1$ has factors in common with $\mathfrak{D}\mathfrak{M}$, this no longer needs to be true at those primes. Furthermore, not all $x$'s satisfying the congruence condition will necessarily appear as $x-$linkings, as this depends on the actual factorization of $\frac{D_1D_2-x^2}{4}$, and not just on congruences. For example, this number will always be divisible by $\mathfrak{D}\mathfrak{M}$, but prime factors of $\mathfrak{D}$ could appear to even powers.

Lemma \ref{lem:xlinkingdivisionamongorientations} allows us to count the sizes of $\Emb_{o_1,o_2}^+(\Ord,D_1,D_2,x,\ell)$, by dividing $\vert\Emb^+(\Ord, D_1, D_2, x, \ell)\vert$ across the total number of orientations. We record this in the final corollary.

\begin{corollary}\label{cor:countembwitheverything}
	Let $B$ be an indefinite quaternion algebra over $\QQ$ of discriminant $\mathfrak{D}$, let $\Ord$ be an Eichler order of level $\mathfrak{M}$, let $(D_1,D_2,x)$ be \nice, and let $\ell$ be a positive integer. Factorize
	\[\dfrac{D_1D_2-x^2}{4}=\pm\prod_{i=1}^{r}p_i^{2e_i+1}\prod_{i=1}^s q_i^{2f_i}\prod_{i=1}^t w_i^{g_i},\]
	where the $p_i$ are the primes for which $\epsilon(p_i)=-1$ that appear to an odd power, $q_i$ are the primes for which $\epsilon(q_i)=-1$ that appear to an even power, and $w_i$ are the primes for which $\epsilon(w_i)=1$. Assume that
	\begin{itemize}
		\item $\mathfrak{D}=\prod_{i=1}^r p_i$;
		\item None of the $p_i$ or $q_i$ are \pbad;
		\item $\mathfrak{M}=\prod_{i=1}^t w_i^{g_i'}$ with $g_i'\leq g_i$ and $\gcd(\mathfrak{M},D_1D_2)=1$;
		\item $\ell=\prod_{i=1}^{r}p_i^{e_i}\prod_{i=1}^s q_i^{f_i}\prod_{i=1}^t w_i^{g_i''}$, where $2g_i''\leq g_i-g_i'$ and $g_i''=0$ if $w_i\mid\PB(D_1,D_2)$.
	\end{itemize}
	Let $n$ be the number of indices $i$ for which $2g_i''<g_i-g_i'$. If $x^2<D_1D_2$, then for every pair of orientations $(o_1,o_2)$, we have
	\[\vert\Emb_{o_1,o_2}^+(\Ord,D_1,D_2,x,\ell)\vert=2^{n}\text{ or }0.\]
	If $x^2>D_1D_2$, then the same result holds without the $+$ and $n$ replaced by $n+1$.
\end{corollary}
\begin{proof}
	By Theorem \ref{thm:maincountingthm}, the count without the orientations or $+$ is $2^{\omega(\mathfrak{D}\mathfrak{M})+n+1}$. If $p\mid\mathfrak{D}\mathfrak{M}$, then since the triple is nice and $\gcd(\mathfrak{M},D_1D_2)=1$, Lemma \ref{lem:xlinkingdivisionamongorientations} implies that there are precisely $2$ pairs $(o_p(\phi_1),o_p(\phi_2))$ which admit $x-$linking. Hence we divide by $2$ for all such $p$, eliminating the factor of $2^{\omega(\mathfrak{D}\mathfrak{M})}$. Finally, if $x^2<D_1D_2$, exactly half of the embeddings have positive sign, which implies the result.
\end{proof}

Corollary \ref{cor:countembwitheverything} approaches the limits of what we can do with this approach. When non-empty, the set $\Emb_{o_1,o_2}^+(\Ord,D_1,D_2,x,\ell)$ has size $2^n$, and distributes itself across the $h^+(D_1)h^+(D_2)$ pairs of equivalence classes of the specified orientations. A rough description of what we can say about this distribution is as follows:

\begin{itemize}
	\item Fix $[(\phi_1,\phi_2)]\in\Emb_{o_1,o_2}^+(\Ord,D_1,D_2,x, \ell)$. Then the map $\theta$ found in Proposition \ref{prop:countemb} combined with the work on $T_{\phi_1,\phi_2}(\mathfrak{M})$ allows us to describe possible values of $\nrd(r)$ for $r\in\Ord$ such that $[(\phi_1^r,\phi_2^r)]\in\Emb_{o_1,o_2}^+(\Ord,D_1,D_2,x, \ell)$ and $[(\phi_1^r,\phi_2^r)]\neq[(\phi_1,\phi_2)]$; they are essentially products of powers of prime divisors $p$ of $\frac{D_1D_2-x^2}{4}$ with $\epsilon(p)=1$.
	\item The integers represented by the element of the class group $\Cl^+(D_i)$ taking $\phi_i$ to $\phi_i^r$ correspond to the norms of elements in $\Ord$ conjugating $\phi_i$ to $\phi_i^r$ (see Sections 4.4 and 4.5 of \cite{JRthe}).
	\item In particular, the distribution relates to the representations of products of primes $p\mid\frac{D_1D_2-x^2}{4}$ with $\epsilon(p)=1$ by binary quadratic forms of discriminants $D_1,D_2$.
\end{itemize}

Of course, even if we could make this more formal and explicit, it does not tell us how the distinct $x$ values interact, which is important for intersection numbers.

\section{Examples}
We present a few examples that illustrate the results of Theorem \ref{thm:maincountingthm} and Corollary \ref{cor:countembwitheverything}. All computations were done in PARI/GP (\cite{PARI}), and the code to replicate these examples can be found in \cite{Qquad}.

\begin{example}\label{ex:1ofcountingformula}
	Let $D_1=5$ and $D_2=381$, so that $D_1,D_2$ are coprime and fundamental. Since $43<\sqrt{5\cdot 381}<44$, to compute which algebras admit non-trivial intersections of $D_1,D_2$, it suffices to compute $\frac{5\cdot 381-x^2}{4}$ for odd $\vert x\vert\leq 43$, and find $\epsilon(p)$ for all prime divisors. The values of $\epsilon(p)$ with $p\leq 80$ are in Table \ref{tab:ex1ep}.
	
	\begin{table}[hbt]
		\begin{center}
			\caption{$\epsilon(p)$ for $D_1=5$, $D_2=381$, $p\leq 80$.}\label{tab:ex1ep}
			\renewcommand{\arraystretch}{1.3}
			\begin{tabular}{|c|c|c|c|c|c|c|c|c|c|c|c|c|c|c|c|} 
				\hline
				$p$ & $2$ & $3$ & $5$ & $7$ & $17$ & $19$ & $29$ & $31$ & $43$ & $47$ & $59$ & $61$ & $67$ & $79$\\ \hline
				$\epsilon(p)$ & $-1$ & $-1$ & $1$ & $-1$ & $-1$ & $1$ & $1$ & $1$ & $-1$ & $-1$ & $1$ & $1$ & $-1$ & $1$\\ \hline
			\end{tabular}
		\end{center}
	\end{table}
	
	Table \ref{tab:ex1possx} displays the possible ramifications of the quaternion algebras, along with the corresponding positive $x$'s (since $x$ and $-x$ correspond to the same algebra).
	
	\begin{table}[htb]
		\begin{center}
			\caption{Quaternion algebras admitting non-trivial intersections in a maximal order for discriminants $5$ and $381$.}\label{tab:ex1possx}
			\renewcommand{\arraystretch}{1.3}
			\begin{tabular}{|c|c|c|c|c|c|c|} 
				\hline
				Ramifying primes & $\emptyset$           & $2,3$          & $2,7$          & $2,17$ & $2,43$ & $2,47$ \\ \hline
				Positive $x$'s   & $7,17,25,31$ & $3,9,21,27,39$ & $13,29,41,43$  & $35$   & $23$ & $5$   \\ \hline\hline
				Ramifying primes & $2,67$ & $2,193$ & $2,223$ & $3,7$ & $3,17$ & $7,17$ \\ \hline
				Positive $x$'s   & $37$   & $19$    & $11$    & $15$  & $33$   & $1$ \\ \hline
			\end{tabular}
		\end{center}
	\end{table}
	
	Let's focus on $B=\left(\frac{3,-1}{\QQ}\right)$, which is ramified at $2,3$. Let $\Ord=\langle1,i,j,\frac{1+i+j+k}{2}\rangle_{\ZZ}$, which is maximal. There are four orientations and $h^+(5)=1$, hence by Proposition \ref{prop:countembD} there are $4$ embedding classes of discriminant $5$. Since $h^+(381)=2$ and $3\mid 381$, there are two orientations, and $4$ total embedding classes of discriminant $381$. Representative embeddings are given in Table \ref{tab:ex1embs}.
	
	\begin{table}[htb]
		\begin{center}
			\caption{Optimal embedding classes for $D=5,381$.}\label{tab:ex1embs}
			\renewcommand{\arraystretch}{1.3}
			\begin{tabular}{ |c|c|c|c|c| } 
				\hline
				$D$ & $o_2(\phi)$ & $o_3(\phi)$ & $\phi\left(\frac{p_D+\sqrt{D}}{2}\right)$ \\ \hline
				$5$ & $1$ & $1$ & $\frac{1-i-j+k}{2}$ \\ \hline
				$5$ & $-1$ & $1$ & $\frac{1-i-j-k}{2}$ \\ \hline
				$5$ & $1$ & $-1$ & $\frac{1+i+j+k}{2}$ \\ \hline
				$5$ & $-1$ & $-1$ & $\frac{1+i+j-k}{2}$ \\ \hline
				$381$ & $1$ & $0$ & $\frac{1-11i-3j+3k}{2}$ \\ \hline
				$381$ & $1$ & $0$ & $\frac{1+9i-3j+7k}{2}$ \\ \hline
				$381$ & $-1$ & $0$ & $\frac{1-11i-3j-3k}{2}$ \\ \hline
				$381$ & $-1$ & $0$ & $\frac{1+9i-3j-7k}{2}$ \\ \hline
			\end{tabular}
		\end{center}
	\end{table}
	
	The possible $x$'s have $\vert x\vert=\{3,9,21,27,39\}$. For each $x$, we factor $\frac{5\cdot 381-x^2}{4}$ in Table \ref{tab:ex1factor}, and determine the possible levels.
	
	\begin{table}[htb]
		\begin{center}
			\caption{Factorization of $\frac{5\cdot 381-x^2}{4}$ for $\vert x\vert=\{3,9,21,27,39\}$.}\label{tab:ex1factor}
			\renewcommand{\arraystretch}{1.3}
			\begin{tabular}{|c|c|c|c|c|c|c|} 
				\hline
				$\vert x\vert$ & $\frac{5\cdot 381-x^2}{4}$ & $\prod p_i^{e_i}$ & $\prod q_i^{f_i}$ & $\prod w_i^{g_i}$ & Possible levels & $n$ \\ \hline
				$3$   & $474$                 & $2^1 3^1$         &                   & $79^1$ & $1$ & $1$ \\ \hline
				$9$   & $456$                 & $2^3 3^1$         &                   & $19^1$ & $2$ & $1$ \\ \hline
				$21$  & $366$                 & $2^1 3^1$         &                   & $61^1$ & $1$ & $1$ \\ \hline
				$27$  & $294$                 & $2^1 3^1$         &  $7^2$            &        & $7$ & $0$ \\ \hline
				$39$  & $96$                  & $2^5 3^1$         &                   &        & $4$ & $0$ \\ \hline
			\end{tabular}
		\end{center}
	\end{table}
	
	It turns out that each $x$ corresponds to a unique level, though this need not be the case in general. This data says that $\vert\Emb_{o_1,o_2}^+(\Ord,5,381,x,\ell)\vert$ should be $0$ or $2$ for $\vert x\vert\in\{3,9,21\}$, and $0$ or $1$ for the $\vert x\vert\in\{27,39\}$. Let $\phi_1$ be the first embedding of discriminant $5$ as given in Table \ref{tab:ex1embs}, and let $\sigma_1,\sigma_2$ be the first two embeddings of discriminant $381$ as given in the same table. For each intersection of $\phi_1$ with $\sigma_i$, we take a pair $(\phi_1',\sigma_i)$ representing the intersection, and record the data in Table \ref{tab:ex1ints} (the signed level is the product of the sign and the level).
	
	\begin{table}[htb]
		\begin{center}
			\caption{Intersection of $\phi_1$ with $\sigma_1,\sigma_2$.}\label{tab:ex1ints}
			\renewcommand{\arraystretch}{1.3}
			\begin{tabular}{|c|c|c||c|c|c|c|} 
				\hline
				\multicolumn{3}{|c||}{Intersections with $\sigma_1$} & \multicolumn{3}{c|}{Intersections with $\sigma_2$} \\ \hline
				$\phi_1'\left(\frac{1+\sqrt{5}}{2}\right)$ & $x$   & Signed level & $\phi_1'\left(\frac{1+\sqrt{5}}{2}\right)$ & $x$ & Signed level \\ \hline
				$\frac{1-13i-55j-29k}{2}$                  & $3$   & $-1$         & $\frac{1+i-j-k}{2}$                        & $3$   & $1$ \\ \hline
				$\frac{1-13i+197j-113k}{2} $               & $3$   & $-1$         & $\frac{1+101i+359j-181k}{2}$               & $3$   & $1$ \\ \hline
				$\frac{1+31i+131j+69k}{2}$                 & $-9$  & $2$          & $\frac{1-i-j+k}{2}$                        & $-9$  & $-2$ \\ \hline
				$\frac{1+31i-469j+269k}{2} $               & $-9$  & $2$          & $\frac{1-41i-145j+73k}{2}$                 & $-9$  & $-2$ \\ \hline
				$\frac{1-87i-373j-197k}{2}$                & $-21$ & $-1$         & $\frac{1+i+5j-3k}{2}$                      & $-21$ & $1$\\ \hline
				$\frac{1-711i-3031j-1599k}{2}$             & $-21$ & $-1$         & $\frac{1+11i+41j-21k}{2}$                  & $-21$ & $1$ \\ \hline
				$\frac{1+223i+953j+503k}{2}$               & $27$  & $7$          & $\frac{1-3i-13j+7k}{2}$                    & $27$  & $-7$ \\ \hline
				$\frac{1-i-j+k}{2}$                        & $39$  & $4$          & $\frac{1-29i+71j+29k}{2}$                  & $39$  & $-4$ \\ \hline
			\end{tabular}
		\end{center}
	\end{table}
	
	This data agrees with the theoretical claim. It also satisfies Proposition \ref{prop:orxsinglemod2dm}, since the $x-$values are all equivalent modulo $2\mathfrak{D}\mathfrak{M}=12$. For the other orientation of $381$, we have essentially the same data, except the $x$'s are all negated.
	
\end{example}

For another interesting example, we consider a non-maximal Eichler order, and compare it to the results for the maximal order.

\begin{example}
	Let $D_1=73$, $D_2=937$, and $x=89$. Then $D_1,D_2$ are coprime, fundamental, and have class number $1$ each. Let $B=\left(\frac{7,5}{\QQ}\right)$, which is ramified at $5,7$. Let $\Ord$ be a maximal order and $\Ord'$ an Eichler order of level $3$, given by
	\[\Ord=\left\langle 1,\frac{1+j}{2},i,\frac{1+i+j+k}{2}\right\rangle_{\ZZ}, \qquad \Ord'=\left\langle 1,i,\frac{1+3j}{2},\frac{1+i+j+k}{2}\right\rangle_{\ZZ}.\]
	There are $4$ embedding classes into $\Ord$ and $8$ embedding classes into $\Ord'$ of each discriminant, each corresponding to a distinct orientation. Since
	\[\dfrac{73\cdot 937-89^2}{4}=(5^1 7^1)()(2^4 3^3),\]
	with $\epsilon(5)=\epsilon(7)=-1$ and $\epsilon(2)=\epsilon(3)=1$ (the empty parentheses indicate the absence of $q_i$'s), the sets $\Emb( X,73,937,89)$ should be non-empty for $X= \Ord, \Ord'$. Fix the optimal embeddings
	\[\phi_1\left(\frac{1+\sqrt{73}}{2}\right)=\frac{1-2i+3j}{2},\qquad\phi_2\left(\frac{1+\sqrt{937}}{2}\right)=\frac{1+14i+5j-4k}{2},\]
	which land in and are optimal with respect to both $\Ord$ and $\Ord'$. Since 
	\[\frac{1}{2}\trd(\phi_1(\sqrt{73})\phi_2(\sqrt{937}))=-121\equiv 89\pmod{2\cdot 3\cdot 5\cdot 7},\]
	$\Int_{\Ord}(\phi_1,\phi_2)$ and $\Int_{\Ord'}(\phi_1,\phi_2)$ should have $89-$linkage. As the class numbers are both one, this is all of the $89-$linkage for the given orientations. Corollary \ref{cor:countembwitheverything} predicts the levels and counts, which is recorded in Table \ref{tab:ex2theorlevel}.
	
	\begin{table}[htb]
		\begin{center}
			\caption{Theoretical prediction for counts of levels.}\label{tab:ex2theorlevel}
			\renewcommand{\arraystretch}{1.3}
			\begin{tabular}{|c|c|c|} 
				\hline
				$\ell$ & $\vert\Emb_{o_1,o_2}^+(\Ord,73,937,89,\ell)\vert$ & $\vert\Emb_{o_1,o_2}^+(\Ord',73,937,89,\ell)\vert$ \\ \hline
				$1$    & $4$                                             & $4$ \\ \hline
				$2$    & $4$                                             & $4$ \\ \hline
				$3$    & $4$                                             & $2$ \\ \hline
				$4$    & $2$                                             & $2$ \\ \hline 
				$6$    & $4$                                             & $2$ \\ \hline
				$12$   & $2$                                             & $1$ \\ \hline
			\end{tabular}
		\end{center}
	\end{table}
	
	The difference in counts comes only at $w_i=3$, where $2g_i''<g_i-g_i'=3-g_i'$ is true for $g_i''=0,1$ when the level is maximal, but is only true for $g_i''=0$ when $g_i'=1$, the Eichler order of level $3$.
	
	We compute the $89-$linkage of $\phi_1,\phi_2$. For each intersection with positive sign, we take a representative pair $(\phi_1,\phi_2')$, and record $\phi_2'$ and the level in Tables \ref{tab:ex2ints1} and \ref{tab:ex2ints2}.
	
	\begin{table}[htb]
		\begin{center}
			\caption{Positive $89-$linking of $\phi_1$ with $\phi_2$ in $\Ord$.}\label{tab:ex2ints1}
			\renewcommand{\arraystretch}{1.3}
			\begin{tabular}{|c|c||c|c|c|} 
				\hline
				$\phi_2'\left(\frac{1+\sqrt{937}}{2}\right)$ & $\ell$ & $\phi_2'\left(\frac{1+\sqrt{937}}{2}\right)$ & $\ell$ \\ \hline
				$\frac{1+22559i+21061j-12851k}{2}$           & $1$    & $\frac{1+119i+117j-69k}{2}$                  & $3$    \\ \hline
				$\frac{1+1769i+1657j-1009k}{2}$              & $1$    & $\frac{1+1428689i+1333449j-813783k}{2}$      & $3$    \\ \hline
				$\frac{1+1769i+1657j+1009k}{2}$              & $1$    & $\frac{1+14i+19j-8k}{2}$                     & $4$    \\ \hline
				$\frac{1+22559i+21061j+12851k}{2}$           & $1$    & $\frac{1+14i+19j+8k}{2}$                     & $4$    \\ \hline
				$\frac{1+584i+551j+334k}{2}$                 & $2$    & $\frac{1+6907484i+6446991j-3934506k}{2}$     & $6$    \\ \hline
				$\frac{1+584i+551j-334k}{2}$                 & $2$    & $\frac{1+4664i+4359j+2658k}{2}$              & $6$    \\ \hline
				$\frac{1+44i+47j+26k}{2}$                    & $2$    & $\frac{1+6907484i+6446991j+3934506k}{2}$     & $6$    \\ \hline
				$\frac{1+44i+47j-26k}{2}$                    & $2$    & $\frac{1+4664i+4359j-2658k}{2}$              & $6$    \\ \hline
				$\frac{1+119i+117j+69k}{2}$                  & $3$    & $\frac{1+179534i+167571j-102264k}{2}$        & $12$   \\ \hline
				$\frac{1+1428689i+1333449j+813783k}{2}$      & $3$    & $\frac{1+179534i+167571j+102264k}{2}$        & $12$   \\ \hline
			\end{tabular}
		\end{center}
	\end{table}
	
	\begin{table}[htb]
		\begin{center}
			\caption{Positive $89-$linking of $\phi_1$ with $\phi_2$ in $\Ord'$.}\label{tab:ex2ints2}
			\renewcommand{\arraystretch}{1.3}
			\begin{tabular}{|c|c||c|c|c|} 
				\hline
				$\phi_2'\left(\frac{1+\sqrt{937}}{2}\right)$ & $\ell$  & $\phi_2'\left(\frac{1+\sqrt{937}}{2}\right)$ & $\ell$ \\ \hline
				$\frac{1+1769i+1657j+1009k}{2}$              & $1$     & $\frac{1+1428689i+1333449j+813783k}{2}$      & $3$    \\ \hline
				$\frac{1+119i+117j+69k}{2}$                  & $1$     & $\frac{1+119i+117j-69k}{2}$                  & $3$    \\ \hline
				$\frac{1+1428689i+1333449j-813783k}{2}$      & $1$     & $\frac{1+14i+19j-8k}{2}$                     & $4$    \\ \hline
				$\frac{1+22559i+21061j-12851k}{2}$           & $1$     & $\frac{1+179534i+167571j+102264k}{2}$        & $4$    \\ \hline
				$\frac{1+44i+47j+26k}{2}$                    & $2$     & $\frac{1+6907484i+6446991j-3934506k}{2}$     & $6$    \\ \hline
				$\frac{1+584i+551j-334k}{2}$                 & $2$     & $\frac{1+4664i+4359j-2658k}{2}$              & $6$    \\ \hline
				$\frac{1+6907484i+6446991j+3934506k}{2}$     & $2$     & $\frac{1+179534i+167571j-102264k}{2}$        & $12$   \\ \hline
				$\frac{1+4664i+4359j+2658k}{2}$              & $2$     &                                              &        \\ \hline
			\end{tabular}
		\end{center}
	\end{table}
	
	This data agrees with Table \ref{tab:ex2theorlevel}.
	
\end{example}

For a final example, we introduce a non-fundamental discriminant.

\begin{example}
	Let $D_1=241$ and $D_2=2736$, which are coprime, and let $x=324$. Note that $D_1$ is fundamental, but $D_2=2^2 3^2 76$, where $76$ is fundamental. Take $B=\left(\frac{77,-1}{\QQ}\right)$, which is ramified at $7,11$. Let $\Ord=\langle 1,\frac{1+i}{2},j,\frac{j+k}{2}\rangle_{\ZZ}$, which is maximal. We have $h^+(241)=1$ and $h^+(2736)=4$, and consider the $5$ optimal embeddings in Table \ref{tab:ex3embs} (one being of discriminant $241$, and the other $4$ being one entire orientation of discriminant $2736$).
	
	\begin{table}[hb]
		\begin{center}
			\caption{Optimal embedding classes for $D=241,2736$.}\label{tab:ex3embs}
			\renewcommand{\arraystretch}{1.4}
			\begin{tabular}{|c|c|c|c|c|c|} 
				\hline
				Label      & $D$ & $o_{7}(\phi)$ & $o_{11}(\phi)$ & $\phi\left(\frac{p_D+\sqrt{D}}{2}\right)$ \\ \hline
				$\phi$     & $241$  & $1$ & $1$ & $\frac{1+i-12j+2k}{2}$   \\ \hline
				$\sigma_1$ & $2736$ & $1$ & $1$ & $\frac{2i-50j+8k}{2}$    \\ \hline
				$\sigma_2$ & $2736$ & $1$ & $1$ & $\frac{10i-281j+31k}{2}$ \\ \hline
				$\sigma_3$ & $2736$ & $1$ & $1$ & $\frac{2i-50j-8k}{2}$    \\ \hline
				$\sigma_4$ & $2736$ & $1$ & $1$ & $\frac{10i-281j-31k}{2}$ \\ \hline
			\end{tabular}
		\end{center}
	\end{table}
	
	Factorize
	\[\dfrac{241\cdot2736-324^2}{4}=(7^1 11^1)()(2^3 3^2 5^2),\]
	where $\epsilon(7)=\epsilon(11)=-1$ and $\epsilon(2)=\epsilon(3)=\epsilon(5)=1$. As $\PB(241,2736)=2\cdot 3$, the primes $2,3$ are \pbad $\,$ and therefore cannot occur in the intersection level. In particular, for $324-$linking, the only valid intersection levels are $1,5$ (whereas if $D_1,D_2$ were fundamental, we could get all divisors of $30$). The table of predicted levels and counts is found in Table \ref{tab:ex3theorlevel}.
	
	\begin{table}[htb]
		\begin{center}
			\caption{Theoretical prediction for counts of levels.}\label{tab:ex3theorlevel}
			\renewcommand{\arraystretch}{1.4}
			\begin{tabular}{|c|c|} 
				\hline
				$\ell$ & $\vert\Emb_{o_1,o_2}^+(\Ord,241,2736,324,\ell)\vert$ \\ \hline
				$1$    & $8$ \\ \hline
				$5$    & $4$ \\ \hline
			\end{tabular}
		\end{center}
	\end{table}
	
	Since 
	\[\frac{1}{2}\trd(\phi(\sqrt{241})\sigma_1(\sqrt{2736}))=786\equiv 324\pmod{2\cdot 7\cdot 11},\]
	intersections of $\phi$ with $\sigma_i$ should exhibit the above $324-$linking behaviour. We compute the possible positive $324-$linking between $\phi$ and $\sigma_i$ for $i=1,2,3,4$, and represent each intersection by a pair $(\phi',\sigma_i)$. The corresponding data is found in Table \ref{tab:ex3ints}.
	
	\begin{table}[htb]
		\begin{center}
			\caption{Positive $324-$linking of $\phi_1$ with $\sigma_i$.}\label{tab:ex3ints}
			\renewcommand{\arraystretch}{1.4}
			\begin{tabular}{|c|c|c||c|c|c|c|} 
				\hline
				$i$ & $\phi_1'\left(\frac{1+\sqrt{241}}{2}\right)$ & $\ell$ & $i$ & $\phi_1'\left(\frac{1+\sqrt{241}}{2}\right)$ & $\ell$ \\ \hline
				$1$ & $\frac{1+51079i+839827j-80937k}{2}$          & $1$ & $3$ & $\frac{1-5i-89j-9k}{2}$                         & $1$ \\ \hline
				$1$ & $\frac{1+39i-397j+23k}{2}$                   & $1$ & $3$ & $\frac{1-449i+4531j+255k}{2}$                   & $1$ \\ \hline
				$1$ & $\frac{1+2433i-24575j+1387k}{2}$             & $5$ & $3$ & $\frac{1-7i+65j+3k}{2}$                         & $5$ \\ \hline
				$2$ & $\frac{1-17i+1220j-138k}{2}$                 & $1$ & $4$ & $\frac{1-87657i+1615987j+161959k}{2}$           & $1$ \\ \hline
				$2$ & $\frac{1+259i-4786j+480k}{2}$                & $1$ & $4$ & $\frac{1-21i+373j+37k}{2}$                      & $1$ \\ \hline
				$2$ & $\frac{1+5i-89j+9k}{2}$                      & $5$ & $4$ & $\frac{1-1395i+25706j+2576k}{2}$                & $5$ \\ \hline
			\end{tabular}
		\end{center}
	\end{table}
	This data agrees with the theoretical claim.
	
\end{example}

\appendix

\section{Hermite normal form calculation}\label{app:hnfcalc}
We calculate the determinant of the row-space of the matrix
\[M=\left(\begin{matrix} 1 & 0 & 0 & 0\\ \frac{p_1}{2} & \frac{1}{2} & 0 & 0\\ \frac{p_2}{2} & 0 & \frac{1}{2} & 0\\0 & 0 & 0 & \frac{1}{2} \\\frac{p_1p_2+x}{4} & \frac{p_2}{4} & \frac{p_1}{4} & \frac{\ell}{4}\\0 & \frac{-x}{4\ell} & \frac{D_1}{4\ell} & \frac{p_1}{4}\\0 & \frac{-D_2}{4\ell} & \frac{x}{4\ell} & \frac{p_2}{4}\\ \frac{x^2-D_1D_2}{8\ell} & \frac{-p_2x-p_1D_2}{8\ell} & \frac{p_1x+p_2D_1}{8\ell} & \frac{p_1p_2+x}{8}\\ \end{matrix}\right),\]
where: 
\begin{itemize}
	\item $D_1,D_2$ are discriminants with parities $p_1,p_2$ respectively;
	\item $\gcd\left(D_1,D_2,D_1D_2-x^2\right)=1$
	\item $4\ell^2\mid D_1D_2-x^2$.
\end{itemize}
Let $L$ be this rowspace, and label the rows $r_1,\ldots,r_8$. Since $L\supseteq\ZZ^4$ and $\ZZ^4$ has determinant $1$, we see that the determinant of $L$ is $\frac{1}{N}$ for some positive integer $N$. Our aim is to show that $N=16\ell$. We can compute $N$ by tensoring our space with $\ZZ_p$ for all primes $p$, and determining the power of $p$ dividing the determinant of the corresponding $\ZZ_p$ lattice.

Note that all denominators of $M$ divide $8\ell$. Hence $p\nmid 2\ell$ implies that $L_p=\ZZ_p^4$, and so $v_p(N)=0$, as desired.

Next, assume that $p\mid 2\ell$ is odd. Thus $p\mid\ell\mid D_1D_2-x^2$, which implies that $D_1$ and $D_2$ are not both divisible by $p$. The first four rows of $M_p$ span $\ZZ_p^4$, and the fifth row is already in this span. Since $\ell\mid\ell^2\mid D_1D_2-x^2$, by removing the powers of $2$ and applying row operations, the last three rows (labeled $r_6',r_7',r_8'$ in order) become
\[\left(\begin{matrix}0 & \frac{-x}{\ell} & \frac{D_1}{\ell} & 0\\0 & \frac{-D_2}{\ell} & \frac{x}{\ell} & 0\\0 & \frac{-p_2x-p_1D_2}{\ell} & \frac{p_1x+p_2D_1}{\ell} & 0\end{matrix}\right).\]
First, $r_8'=p_2r_6'+p_1r_7'$, so we can ignore $r_8'$. Next, we have 
\[xr_6'-D_1r_7',D_2r_6'-xr_7'\in\langle r_1,r_2,r_3,r_4\rangle_{\ZZ_p}.\]
Without loss of generality assume that $p\nmid D_1$, whence $r_7'\in\langle r_1,r_2,r_3,r_4,r_6'\rangle_{\ZZ_p}$. Then $r_3\in\langle r_1,r_2,r_4,r_6'\rangle_{\ZZ_p}$, and thus our basis is spanned by
\[\left(\begin{matrix}1 & 0 & 0 & 0\\0 & 1 & 0 & 0\\0 & \frac{-x}{\ell} & \frac{D_1}{\ell} & 0\\ 0 & 0 & 0 & 1\end{matrix}\right).\]
The power of $p$ dividing the denominator of this determinant is $v_p(\ell)=v_p(16\ell)$, as desired.

The remaining case is $p=2$. Let $v_2(\ell)=k\geq 0$, write $\ell=2^k\ell'$ with $\ell'$ odd, and without loss of generality, assume that $D_1$ is odd. Working over $\ZZ_2$, we multiply out by odd factors to obtain the row-space
\[M_1=\left(\begin{matrix} 1 & 0 & 0 & 0\\ \frac{1}{2} & \frac{1}{2} & 0 & 0\\ \frac{p_2}{2} & 0 & \frac{1}{2} & 0\\0 & 0 & 0 & \frac{1}{2} \\\frac{p_2+x}{4} & \frac{p_2}{4} & \frac{1}{4} & \ell'2^{k-2}\\0 & \frac{-x}{2^{k+2}} & \frac{D_1}{2^{k+2}} & \frac{\ell'}{4}\\0 & \frac{-D_2}{2^{k+2}} & \frac{x}{2^{k+2}} & \frac{p_2\ell'}{4}\\ \frac{x^2-D_1D_2}{2^{k+3}} & \frac{-p_2x-D_2}{2^{k+3}} & \frac{x+p_2D_1}{2^{k+3}} & \frac{(p_2+x)\ell'}{8}\\ \end{matrix}\right).\]
We now find the span of the first $5$ rows, and successively add in rows $6$ through $8$ in the various cases.
\begin{itemize}
	\item If $D_2$ is even,
	\begin{itemize}
		\item If $k=0$, 
		\begin{itemize}
			\item If $2\mid\mid x$, rows $1$ to $5$ give
			\[\left(\begin{matrix}\frac{1}{2} & \frac{1}{2} & 0 & 0\\0 & \frac{1}{2} & \frac{1}{4} & \frac{1}{4}\\0 & 0 & \frac{1}{2} & 0\\0 & 0 & 0 & \frac{1}{2}\end{matrix}\right).\]
			Rows $6$ and $7$ already lie in this span, and row $8$ shifts to $\left(\begin{smallmatrix}\frac{x^2-D_1D_2}{8} & \frac{-D_2}{8} & \frac{1}{4} & \frac{1}{4}\end{smallmatrix}\right)$. If $4\mid\mid D_2$, it follows that $8\mid D_1D_2-x^2$, and after a $\ZZ_2^4$ shift, row $8$ becomes $\left(\begin{smallmatrix}0 & \frac{1}{2} & \frac{1}{4} & \frac{1}{4}\end{smallmatrix}\right)$, which is already in the span. Otherwise, $8\mid D_2$, and by a $\ZZ_2^4$ shift we arrive at $\left(\begin{smallmatrix}\frac{1}{2} & 0 & \frac{1}{4} & \frac{1}{4}\end{smallmatrix}\right)$. Thus rows $1$ through $5$ sufficed, we get the determinant $2^{-4}$, so the power of two dividing the denominator is $4=k+4$, as desired.
			\item If $4\mid x$, rows $1$ to $5$ give
			\[\left(\begin{matrix}\frac{1}{2} & \frac{1}{2} & 0 & 0\\0 & 1 & 0 & 0\\0 & 0 & \frac{1}{4} & \frac{1}{4}\\0 & 0 & 0 & \frac{1}{2}\end{matrix}\right),\]
			and the last three rows already lie in this span. The determinant is again $2^{-4}$, as desired.
		\end{itemize}
		\item If $k=1$, then $16\mid D_1D_2-x^2$.
		\begin{itemize}
			\item If $2\mid\mid x$, then $4\mid\mid D_2$ necessarily. The first $5$ rows give
			\[\left(\begin{matrix}\frac{1}{2} & \frac{1}{2} & 0 & 0\\0 & \frac{1}{2} & \frac{1}{4} & 0\\0 & 0 & \frac{1}{2} & 0\\0 & 0 & 0 & \frac{1}{2}\end{matrix}\right).\]
			Shifting the sixth row gives $\left(0,\pm\frac{1}{4},\frac{1}{8},\frac{1}{4}\right)$ (using $D_1\equiv 1\pmod{4}$), which can replace row two, giving
			\[\left(\begin{matrix}\frac{1}{2} & \frac{1}{2} & 0 & 0\\0 & \pm\frac{1}{4} & \frac{1}{8} & \frac{1}{4}\\0 & 0 & \frac{1}{2} & 0\\0 & 0 & 0 & \frac{1}{2}\end{matrix}\right).\]
			The seventh and eighth rows lie in this span, and the determinant is $2^{-5}$, as desired.
			\item If $4\mid x$, then $16\mid D_2$ necessarily. The first $5$ rows give
			\[\left(\begin{matrix}\frac{1}{2} & \frac{1}{2} & 0 & 0\\0 & 1 & 0 & 0\\0 & 0 & \frac{1}{4} & 0\\0 & 0 & 0 & \frac{1}{2}\end{matrix}\right).\]
			Rows $7$ and $8$ already lie in this span, and row $6$ shifts to $\left(0,\frac{-x}{8},\frac{1}{8},\frac{1}{4}\right)$. If $4\mid\mid x$ we can replace the second row, and if $8\mid x$ we can replace the third row, giving
			\[\left(\begin{matrix}\frac{1}{2} & \frac{1}{2} & 0 & 0\\0 & \pm\frac{1}{2} & \frac{1}{8} & \frac{1}{4}\\0 & 0 & \frac{1}{4} & 0\\0 & 0 & 0 & \frac{1}{2}\end{matrix}\right)\text{ and }\left(\begin{matrix}\frac{1}{2} & \frac{1}{2} & 0 & 0\\0 & 1 & 0 & 0\\0 & 0 & \frac{1}{8} & \frac{1}{4}\\0 & 0 & 0 & \frac{1}{2}\end{matrix}\right).\]
			respectively. This gives determinant $2^{-5}$, as desired.
		\end{itemize}
		\item If $k\geq 2$, then $64\mid 2^{2k+2}\mid D_1D_2-x^2$. The last three rows shift to
		\[\left(\begin{matrix}0 & \frac{-x}{2^{k+2}} & \frac{D_1}{2^{k+2}} & \frac{\ell'}{4}\\0 & \frac{-D_2}{2^{k+2}} & \frac{x}{2^{k+2}} & 0\\0 & \frac{-D_2}{2^{k+3}} & \frac{x}{2^{k+3}} & \frac{x\ell'}{8}\end{matrix}\right).\]
		Since $xr_6-D_1r_7$ lies in the span of the first four rows, so we can eliminate $r_7$ from consideration. Similarly, $\frac{x}{2}r_6-D_1r_8$ also lies in this span, so we can eliminate $r_8$ from consideration too; only the first $6$ rows are left.
		\begin{itemize}
			\item If $2\mid\mid x$, rows $1$ to $5$ give us
			\[\left(\begin{matrix}\frac{1}{2} & 0 & \frac{1}{4} & 0\\0 & \frac{1}{2} & \frac{1}{4} & 0\\0 & 0 & \frac{1}{2} & 0\\0 & 0 & 0 & \frac{1}{2}\end{matrix}\right).\]
			We can replace $r_2$ with $r_6$ giving
			\[\left(\begin{matrix}\frac{1}{2} & 0 & \frac{1}{4} & 0\\0 & \frac{-x}{2^{k+2}} & \frac{D_1}{2^{k+2}} & \frac{\ell'}{4}\\0 & 0 & \frac{1}{2} & 0\\0 & 0 & 0 & \frac{1}{2}\end{matrix}\right),\]
			which has determinant $\frac{-x}{2^{k+5}}$, as desired (since $v_2(x)=1$).
			\item If $4\mid x$, rows $1$ to $5$ give us
			\[\left(\begin{matrix}\frac{1}{2} & \frac{1}{2} & 0 & 0\\0 & 1 & 0 & 0\\0 & 0 & \frac{1}{4} & 0\\0 & 0 & 0 & \frac{1}{2}\end{matrix}\right).\]
			In this case we can replace $r_3$ with $r_6$, giving
			\[\left(\begin{matrix}\frac{1}{2} & \frac{1}{2} & 0 & 0\\0 & 1 & 0 & 0\\0 & \frac{-x}{2^{k+2}} & \frac{D_1}{2^{k+2}} & \frac{\ell'
				}{4}\\0 & 0 & 0 & \frac{1}{2}\end{matrix}\right),\] 
			which has determinant $\frac{D_1}{2^{k+4}}$, as desired (since $D_1$ is odd).
		\end{itemize}
	\end{itemize}
	\item If $D_2$ is odd,
	\begin{itemize}
		\item If $k=0$, then the first $5$ rows give us
		\[\left(\begin{matrix}\frac{1}{2} & \frac{1}{2} & 0 & 0\\0 & \frac{-x}{4} & \frac{1}{4} & \frac{1}{4}\\0 & 0 & 1 & 0\\0 & 0 & 0 & \frac{1}{2}\end{matrix}\right).\]
		The last three rows lie in this span, so we get determinant $2^{-4}$, as desired.
		\item If $k\geq 1$, the first five rows give
		\[\left(\begin{matrix}\frac{1}{2} & \frac{1}{2} & 0 & 0\\0 & \frac{1}{4} & \frac{-x}{4} & 0\\0 & 0 & 1 & 0\\0 & 0 & 0 & \frac{1}{2}\end{matrix}\right).\]
		The second row can be replaced by the seventh, giving
		\[\left(\begin{matrix}\frac{1}{2} & \frac{1}{2} & 0 & 0\\0 & \frac{-D_2}{2^{k+2}} & \frac{x}{2^{k+2}} & \frac{1}{4}\\0 & 0 & 1 & 0\\0 & 0 & 0 & \frac{1}{2}\end{matrix}\right).\]
		This span also contains the sixth and eighth rows, hence is a valid basis. The $2-$adic valuation of this determinant is $-(k+4)$, as desired.
	\end{itemize}
\end{itemize}

\bibliographystyle{alpha}
\bibliography{../references}
\end{document}